\documentclass[11pt]{amsart}
\usepackage{amsmath,amssymb,color}
\usepackage[margin=2.7cm]{geometry}
\usepackage[hidelinks]{hyperref}
\usepackage{graphicx}
\usepackage{mathtools}
\usepackage{enumitem}

\allowdisplaybreaks

\usepackage{tikz}
\usepackage{pgfplots}
\pgfplotsset{compat=1.18}

\newcommand{\N}{{\mathbb N}}
\newcommand{\R}{{\mathbb R}}

\newcommand{\D}{\mathcal{D}}
\newcommand{\E}{\mathcal{E}}

\newcommand{\M}{\mathcal{M}}
\newcommand{\Pcal}{\mathcal{P}} 
\newcommand{\T}{\mathcal{T}}

\newcommand{\dd}{\mathrm{d}}

\newcommand{\Eint}{\E_{{\rm int}}}
\newcommand{\Eext}{\E_{{\rm ext}}}
\newcommand{\Eextz}{\E_{{\rm ext},0}}
\newcommand{\EKL}{\E_{KL}}
\newcommand{\EKpLp}{\E_{K'L'}}
\newcommand{\EKi}{\E_{K_{i-1}K_i}}

\newcommand{\nKs}{n_{K\sigma}}
\newcommand{\Bd}{|B_d|}
\newcommand{\Bdm}{|B_{d-1}|}
\newcommand{\diam}{\mathrm{diam}}
\newcommand{\cN}{\mathcal{N}}

\usepackage[dvipsnames]{xcolor}
\mathtoolsset{showonlyrefs}

\DeclareMathOperator{\argmin}{\mathrm{argmin}}

\title[Discrete Sobolev inequalities]{Discrete Sobolev and Trudinger inequalities\\ on polyhedral meshes}

\newtheorem{theorem}{Theorem}
\newtheorem{lemma}[theorem]{Lemma}
\newtheorem{proposition}[theorem]{Proposition}
\newtheorem{remark}[theorem]{Remark}

\author[M. Herda]{Maxime Herda}
\address{Univ. Lille, CNRS, Inria, UMR 8524--Laboratoire Paul Painlev\'e, 59000 Lille}
\email{maxime.herda@inria.fr}

\author[A. Trescases]{Ariane Trescases}
\address{Institut de Math\'ematiques de Toulouse; UMR5219 - Université de Toulouse ; CNRS - UPS, F-31062 Toulouse Cedex 9, France}
\email{ariane.trescases@math.univ-toulouse.fr}

\author[A. Zurek]{Antoine Zurek}
\address{Universit\'e de Technologie de Compi\`egne, LMAC, 60200 Compi\`egne, France}
\email{antoine.zurek@utc.fr}

\date{\today}

\begin{document}

\begin{abstract}
We revisit the derivation of discrete Sobolev and related inequalities on bounded domains for piecewise constant functions on polyhedral meshes. Our results improve the dependence of the Sobolev constant on the Lebesgue exponents and allow us to recover the asymptotics of the continuous case. As a consequence we derive a discrete version of the Trudinger inequality and variants. We extend our results to broken Sobolev spaces and discuss implications of these new or refined discrete functional inequalities in the numerical analysis of non-conforming numerical methods.

\bigskip
		

\medskip
		
\end{abstract}

\maketitle

\tableofcontents

\section{Introduction}
\subsection{Objectives}

This paper introduces a new method, inspired by~\cite{Trudinger67,GlitzkyGriepentrog10}, to establish Sobolev and related inequalities for piecewise constant and, more generally, piecewise $W^{1,p}$ functions on a polyhedral mesh of a polyhedral bounded domain $\Omega \subset \R^d$ ($d \geq 2$). This type of function arises naturally in the study of finite volume schemes and higher-order non-conforming numerical methods. In contrast with prior works, one of the main novelties of this paper is to derive the expected asymptotic dependence of the constants with respect to Lebesgue exponents. This leads to the first derivation of a discrete counterpart of the Trudinger inequality. 

Let us recall that, for a bounded domain $\Omega$ satisfying a cone condition, see Section~\ref{subsec: domain}, one has~\cite{GilbargTrudinger2001,AdamsFournier2003} that
\begin{align}\label{Sob ineg continuous}
\|u\|_{L^q(\Omega)} \leq C(p,q) \|\nabla u\|_{L^p(\Omega)}, \quad \forall u \in W^{1,p}_0(\Omega),
\end{align}
where the constant $C(p,q)$ is finite in the following cases and behaves asymptotically as 
\begin{align}\label{Constant Sobolev asymptotic}
C(p,q) = \left\{
    \begin{array}{ll}
         O_{p\to d^+}\left((p-d)^{-\left(1-1/d\right)}\right) & \mbox{if }p\in (d,\infty], \, \, q \in [p,\infty], \\
         O_{q\to \infty}\left(q^{1-1/d}\right) & \mbox{if }p=d, \, \, q \in [p,\infty), \\
         O_{p\to d^-}\left((d-p)^{-\left(1-1/d\right)}\right) & \mbox{if }p \in [1,d), \, \, q \in [p,dp/(d-p)].
    \end{array}
\right.
\end{align}
In the super-critical case $p>d$, also referred to as Morrey's inequality, the asymptotic behavior of $C(p,q)$ can be found in~\cite{ErcolePereira2020}. In the critical case $p=d$ as $q$ approaches $\infty$, this behavior is a consequence of the Trudinger inequality~\cite{Trudinger67, Moser1971}. Indeed, while $W^{1,d}_0$ functions might not be essentially bounded, Trudinger inequality implies that for small enough $\alpha>0$ there exists $C(\Omega) > 0$ such that 
\begin{equation}\label{eq:trudinger}
\sup_{\substack{u\in W^{1,d}_0(\Omega)\\\|\nabla u\|_{L^d(\Omega)}\leq 1}}\int_\Omega \exp(\alpha |u|^{\frac{d}{d-1}})\dd x \leq C(\Omega).
\end{equation}
Finally, the asymptotic behavior of $C(p,q)$ in the sub-critical case $p<d$ can be found explicitly in~\cite{Aubin76,Talenti76}. For the sake of brevity, we wrote the inequalities above for functions vanishing on $\partial\Omega$, but their counterparts without conditions on the boundary are equally essential (Sobolev embeddings, Poincaré--Sobolev inequalities). We emphasize that in this paper we establish discrete Sobolev inequalities for general boundary values.

Sobolev inequalities, and the fine dependence of Sobolev constants on the Lebesgue exponents, are crucial tools for the theoretical analysis of PDEs. For instance, Trudinger's proof of the critical embedding of $W^{1,d}_0(\Omega)$ into exponential Orlicz spaces relies on the asymptotic growth $C(d,q)\sim  q^{1-1/d}$ as $q \to \infty$, see~\cite{Trudinger67,Moser1971}. This embedding has been used later to derive a critical threshold leading to blow-up phenomena  for Keller--Segel systems, see~\cite{NagaiSenbaYoshida97,GajZac98,FujieJiang20,JinWang2020}. Moreover, considerations similar to those in~\cite{Trudinger67} underlie the Br\'ezis--Gallou\"et--Wainger type of inequality~\cite{BrezisGallouet80,BrezisWainger80}, which plays a central role in the analysis of some critical nonlinear PDEs. Of course, discrete counterparts of Sobolev inequalities, and more generally discrete functional inequalities, also play a key role in the numerical analysis of finite volume schemes for PDEs. Indeed, they are, for instance, used to prove convergence and sometimes error estimates in the following (non-exhaustive) list of contributions~\cite{Herbin95,EGH00,CoudiereGallouetHerbin2001, CCLP03,AndreianovGutnicWittbold2004,AndreianovBoyerHubert2007,DrbMik2008,CoudiereHubert2011,JuengelZurek2021}. They are also used to study qualitative properties of discrete solutions in~\cite{BessemoulinChatardFilbet2012,BessemoulinChainais2016,ChainaisJuengelSchuchnigg16,CancesChainaisHerdaKrell2020,HTZ25}. Besides, these discrete inequalities play a crucial role in the study of more general non-conforming methods, see for instance~\cite{Brenner2004,Vohralik2005,CreuseNicaise2006,DiPietroErn2010,DroniouEtAl2018} and references therein.

As described hereafter, although numerous discrete Sobolev inequalities are available in the literature, none of the available results, to the best of our knowledge, provides the dependence of the embedding constants on the Lebesgue exponents given by~\eqref{Constant Sobolev asymptotic}. Establishing such inequalities, with the correct asymptotic dependence~\eqref{Constant Sobolev asymptotic} of the constants, thus appears as a natural prerequisite for sharpening the numerical analysis toolbox of non-conforming methods. This is the main purpose of the present paper. Along the way, our analysis  recovers the whole range of expected Lebesgue exponents and fills some gaps in the literature. More precisely, we prove first, for piecewise constant functions defined on a polyhedral mesh of a bounded polyhedral domain $\Omega$:
\begin{itemize}
\item Discrete Sobolev inequalities where the Sobolev constants satisfy~\eqref{Constant Sobolev asymptotic}, see Theorem~\ref{thm:Sobembedding}.
\item Discrete Poincar\'e-Sobolev inequality for zero mean value functions, see Theorem~\ref{thm:PoincarSob}.
\item Discrete Poincar\'e-Sobolev inequality for functions ``vanishing'' on a part of $\partial \Omega$, see Theorem~\ref{thm:PoincarSobDir}.
\item Discrete Gagliardo--Nirenberg--Sobolev inequality for general boundary conditions, see Theorem~\ref{thm:GagliardoNirenbergGeneral}.
\item A new discrete Trudinger inequality, established as a consequence of the improved asymptotics of Sobolev constants, see~Theorem~\ref{thm:Trudinger1}.
\item Some applications of these results, see Section~\ref{sec:discussions}.
\end{itemize}
Moreover, the new method introduced in this paper can be extended to more general broken Sobolev spaces in order to obtain: 
\begin{itemize}
\item All the above discrete functional inequalities for general piecewise Sobolev functions, see Theorem~\ref{thm:DG} and Remark~\ref{rk:DG}.
\end{itemize}
We emphasize that a discrete version of Trudinger inequality, uniform in the mesh size for piecewise constant functions on general meshes, has not been available in the literature, and was the initial motivation of the present paper. Indeed, Trudinger's original idea is based on a series expansion of the exponential function together with the use of Sobolev inequalities~\cite{Trudinger67}. In particular, in this proof, it is crucial to obtain precise asymptotic behavior of the Sobolev constants on the Lebesgue exponents. The asymptotics obtained in the present work make it possible to adapt this strategy at the discrete level and yield an embedding of the broken Sobolev space $W^{1,d}(\Omega)$ into exponential Orlicz spaces.

\subsection{Methods of proof}

Before describing our strategy and since the core of our approach is to establish discrete Sobolev inequalities for piecewise constant functions defined on a polyhedral mesh of a bounded polyhedral domain $\Omega$, we briefly review the methods introduced in the literature to obtain such inequalities.

We can list three main and distinct, although sometimes related, approaches. To the best of our knowledge, the first one seems to go back to Herbin~\cite{Herbin95}, where the author proved in the case $d=2$, for triangular meshes of a bounded polygonal domain $\Omega$, and for Dirichlet boundary conditions, the standard Poincar\'e inequality, i.e., inequality~\eqref{Sob ineg continuous} with $p=q=2$. The underlying idea is to adapt to the discrete setting the representation formula
\begin{align}\label{Path historic}
u(x) = u(y) + \int_0^1 \nabla u(y + t(x-y)) \cdot (x-y) \, \dd t,
\end{align}
for $x,\, y \in \Omega$ such that the segment $[x,y]$ is included in $\Omega$. For this purpose, the line integral of the gradient along a straight path is replaced by a telescopic sum of jumps across successive cells of the mesh intersected by the path. This discrete path argument naturally introduces, up to mesh-dependent geometric factors, a finite volume discrete gradient (see \eqref{def discrete gradient}-\eqref{eq:SobolevSeminormDirichlet} and Lemma~\ref{lem:equivseminorm}). Later, this method has been extended, to unstructured two-dimensional meshes~\cite{CoudiereVilaVilledieu99}, to dimensions up to three on admissible meshes~\cite{EGH00}, to the non-Hilbertian setting~\cite{DroniouGallouetHerbin2003,AndreianovGutnicWittbold2004}, and to the case of general boundary conditions~\cite{EGH00,GallouetHerbinVignal2000}, see also~\cite{BCCHF} for further references.

The second idea is based on Gagliardo and Nirenberg's proof of Sobolev inequalities, which relies on the embedding of $BV(\Omega)$ into $L^{d/(d-1)}(\Omega)$ if $\Omega$ is a Lipschitz bounded domain of $\R^d$. This framework is natural for finite volumes which are conforming in $BV$, since piecewise constant functions on polyhedral meshes belong to this functional space. It was first implemented by adapting the continuous arguments at the discrete level~\cite{CoudiereGallouetHerbin2001,EGH00,DroniouGallouetHerbin2003,ChainaisDroniou2011}. The continuous $BV$ embedding was then used directly by Filbet~\cite{Filbet2006}, and this route was further developed in~\cite{SUSHI2010,AndreianovBendahmaneRuiz2011,BouchutEymardPrignet2011,BCCHF,NgwammouNdjinga2024} and extended to broken polynomial spaces in~\cite{DiPietroErn2010}. However, the application of this technique has some limitations, and fails to recover the boundary case $q=\infty$, in the super-critical regime $p>d$, and, in the critical case $p=d$, it yields $C(d,q)=O(q)$ instead of $O(q^{1-1/d})$.
 
The third idea to prove discrete Sobolev inequalities in the case of Vorono\"i meshes and for general boundary conditions has been proposed by Glitzky and Griepentrog~\cite{GlitzkyGriepentrog10}. Their method is based on a discrete counterpart of Sobolev's integral representation \cite[Theorem~4]{Burenkov},
\begin{align*}
u(x) = \int_\Omega u(y) w(y) \dd y + \int_\Omega \frac{\nabla u(\xi) \cdot (x-\xi)}{|x-\xi|^d} \int_{|x-\xi|}^\infty w\left(x+s \frac{\xi-x}{|\xi-x|}\right) s^{d-1}\, \dd s\, \dd \xi,
\end{align*}
where $w$ is an adequate bump function, and where the cone with vertex $x\in\Omega$ generated by the support of $w$ is included in $\Omega$. This identity is combined with a special treatment of weakly singular integral operators. In this respect, their approach is close to the one proposed in~\cite{Trudinger67,GilbargTrudinger2001} and is a first step towards the representation of a piecewise constant function in terms of the Riesz potential of its discrete gradient, although this terminology is not used in~\cite{GlitzkyGriepentrog10}. We note that this strategy also starts from a path representation of the form~\eqref{Path historic}. Contrary to the approach initiated in~\cite{Herbin95}, the subsequent analysis is based on a refined treatment of the geometric quantities appearing in this representation, leading to a discrete analogue of Sobolev's integral representation. 

In addition to these finite-volume approaches, discrete functional inequalities for broken Sobolev spaces, which contain piecewise constant functions, can also be recovered efficiently either by comparison with conforming reconstructions or by integration-by-parts techniques \cite{brenner03, LasisSuli03, Brenner2004,buffa09,BottiMascotto2026}, though the whole range of indices for Sobolev type inequalities has not been established through these techniques. 

Despite the variety of methods, to the best of our knowledge, none of the existing approaches has been shown to recover the asymptotic behavior~\eqref{Constant Sobolev asymptotic}, cover the full range of exponents under which Sobolev and related inequalities hold, and derive Trudinger type inequalities.

In this paper, we propose to revisit the strategies of~\cite{Herbin95,GlitzkyGriepentrog10} and to establish a new discrete potential inequality of the form
\begin{align}\label{Ineg: discrete potential}
|u(x)| \lesssim \|u\|_{L^1(\Omega)} +  \int_\Omega \frac{|\nabla u|(\xi)}{|x-\xi|^{d-1}} \dd \xi, \quad x \in \Omega.
\end{align}
In order to prove such an estimate, we adapt at the discrete level the method considered in~\cite{Trudinger67}. In the latter, the cone condition on $\Omega$ allows Trudinger to obtain, by integrating along rays inside a cone, a representation of a function thanks to the Riesz potential of its gradient. The adaptation of Trudinger's integral representation argument to the discrete finite volume setting constitutes one of the key novelties of the present paper. Besides, as a by-product, this new approach allows one to establish functional inequalities in (general) broken Sobolev spaces. We refer to~\cite{BottiMascotto2026} and references therein for a clear overview of different strategies introduced in the literature to prove such inequalities in this framework.

Finally, another source of novelty, compared, for instance, to~\cite{Herbin95,EGH00,GlitzkyGriepentrog10,BCCHF}, are the assumptions made on the mesh of $\Omega$. When establishing discrete counterparts of inequalities such as~\eqref{Ineg: discrete potential}, then~\eqref{Sob ineg continuous} and finally~\eqref{eq:trudinger}, we look for constants in the inequalities that are (obviously) independent of the choice of the mesh and (more importantly) of the mesh size. In this respect our work deviates from the many available results about functional inequalities on graphs. Here, we will consider general meshes not necessarily admissible in the sense of~\cite[Definition 9.1]{EGH00}. Instead, the cells are open polyhedral subsets, not necessarily convex, but all star-shaped with respect to a ball having a radius proportional to the mesh size, and whose faces shrink at most like a hypersphere of radius proportional to the mesh size: see Section~\ref{sec: not and results} for the formal statement and Section~\ref{sec:discussion_regularity} on the relaxation of those hypotheses. Hence, the regularity of the mesh will be quantified by two scalar parameters, and all constants in the inequalities will depend solely on these parameters as well as $\Omega$ and (explicitly) on the Lebesgue exponents.

\subsection{Outline of the paper} The paper is organized as follows. Section~\ref{sec: not and results} is dedicated to the formulation of the main assumptions on the domain $\Omega$, the associated mesh, and the presentation of our main results for piecewise constant functions. Then, we prove in Section~\ref{sec: discrete potential inequality} a discrete potential inequality of the form~\eqref{Ineg: discrete potential} which allows us to show in Section~\ref{sec: proofs main results} our main results for piecewise constant functions. Section~\ref{sec:HFV} is dedicated to the generalization to (general) broken Sobolev spaces. In particular, this implies extensions of our results to hybrid and arbitrary high-order methods. In Section~\ref{sec:discussions}, we discuss an application of the discrete Trudinger inequality to the discretization of the Keller--Segel system, as well as complementary results on the optimality of the Moser--Trudinger constant and quantification of discrete Sobolev non-embeddings. Finally, the appendix contains complementary results and discussions regarding the regularity assumptions on the mesh as well as the definition of discrete Sobolev seminorms.

\section{Notation and main results}\label{sec: not and results}

\subsection{Domain, meshes and notation}
\subsubsection{Domain}\label{subsec: domain}

Let $\Omega$ be a connected and bounded polyhedral open subset of $\R^d$, with $d\ge2$. By polyhedral domain, we mean the following: first, a polyhedron is a bounded intersection of a finite number of half-spaces. A polyhedral domain is an open domain composed of a finite union of convex polyhedra, in the sense that its closure is equal to the closure of this union. In particular $\Omega$ is a bounded, possibly non-convex, Lipschitz domain.

Since $\Omega$ is bounded and Lipschitz, it furthermore satisfies a uniform cone condition: there exists a circular (open, finite) cone $k_\Omega$ such that for any $x\in\overline{\Omega}$, there is a congruent cone with vertex $x$, $k_\Omega(x)$, such that $k_\Omega(x)\subset \Omega$, see \cite[Theorem~1.2.2.2]{grisvard11}.

Throughout this paper, we write $B(a,r)$ the open ball centered in $a\in\R^d$ and of radius $r\in[0,\infty)$ in $\R^d$, and for $d'\in\{1,\dots,d\}$ we write $B_{d'}$ the unit open ball in $\R^{d'}$. For any $x,y\in\R^d$ we write $[x,y]=\{t x + (1-t) y,\, t\in [0,1]\}$ the line segment between $x$ and $y$ and $(x,y)=\{t x + (1-t) y,\, t\in \R\}$ the line (or possibly a point) passing through $x$ and $y$. For any set $E\subset \R^{d}$, we write $\overline E$ the closure of $E$, and $\mathring E$ the interior of $E$. For any open set $E\subset \R^{d'}$, where $d'\in\{1,\dots,d\}$, we write $\partial E=\overline{E}\setminus E$ its boundary in $\R^{d'}$. Finally, $|\cdot|$ refers, with a slight abuse of notation, either to the Hausdorff measure of dimension $d$ or $d-1$ on $\R^d$ or the Euclidean norm, and $\dd(\cdot,\cdot)$ denotes the Euclidean distance between points or sets.

\subsubsection{Mesh and notation} A mesh $\M=(\T,\E,\Pcal)$ is composed of:\smallskip
\begin{itemize}
    \item $\T$: a set of open connected polyhedral subsets of $\Omega$ called cells (or control volumes);\smallskip
    \item $\E$: a set of non-empty relatively open connected polyhedral subsets of hyperplanes in $\R^d$, which are called the faces of the control volumes;\smallskip
    \item $\Pcal = \left\{x_K \right\}_{K\in \T}$: a set of points of $\Omega$, indexed by $\T$, which are called the cell centers;\smallskip
\end{itemize}
and satisfies the following properties: \smallskip
\begin{enumerate}[label=(HM.{\alph*})]
    \item \label{HM a partition} The cells cover the domain, that is,
    \begin{align*}
        \overline{\Omega} = \bigcup_{K\in\T} \overline{K}.
    \end{align*}
    \item \label{HM b boundaries} For all $K\in \T$, there exists a subset $\E_K\subset \E$ such that $\partial K = \bigcup_{\sigma\in\E_K} \overline{\sigma}$ and, for all $(\sigma,\sigma')\in \E_K ^2$ with $\sigma\neq\sigma'$, we have $|\sigma \cap \sigma'|=0$. Furthermore,
    \begin{equation*}
    \E=\bigcup_{K\in\T} \E_K.
    \end{equation*}
    \item \label{HM c interior faces}  Interior faces: for all $(K,L)\in\T^2$ with $K\neq L$, we have either $\left|\overline{K}\cap\partial L\right| = \left|\partial K\cap \overline{L}\right|=0$ or $\overline{K}\cap\overline{L}=\bigcup_{\sigma\in \EKL}\overline{\sigma}$ for the non-empty subset $\EKL = \E_K\cap\E_L$ of $\E$, in which case we say that $K$ and $L$ are neighboring cells. We write
    \begin{equation*} 
    \Eint=\{\sigma\in\E\,:\, \exists (K,L)\in\T^2,\, \sigma \in \EKL \}.
    \end{equation*}
    \item \label{HM d exterior faces} Exterior faces: for all $K\in\T$, we have either $\left|\overline{K}\cap\partial\Omega\right|=0$ or $\overline{K}\cap\partial\Omega=\bigcup_{\sigma\in\E_K \cap \Eext} \overline{\sigma}$
    where
    \begin{equation*} \Eext=\{\sigma\in\E\,:\, \sigma \subset \partial \Omega  \}.
    \end{equation*}
    \item \label{HM e cell centers} For all $K\in\T$, $x_K\in K$.\smallskip
\end{enumerate}

We then define, for all $K \in \T$ and $\sigma\in\E_K$,
\begin{equation}
    \dd_\sigma = \begin{cases} |x_K-x_L| \quad &\text{if } \sigma \in \Eint, \,\,\text{ with } \sigma \in \EKL, \\ \dd(x_K,H_\sigma) \quad &\text{if } \sigma \in \Eext,
    \end{cases}
\end{equation}
where $H_\sigma$ is the hyperplane generated by $\sigma$. Observe that the assumptions~\ref{HM c interior faces} and \ref{HM e cell centers} imply that $\dd_\sigma>0$ for all $\sigma\in\Eint$. For any cell $K$ and $\sigma\in \E_K$ we define $\nKs$ as the outward unit normal to $K$ on $\sigma$. We define respectively the cell and the face sizes as
\begin{equation}
    h_K = \diam(K), \quad K\in \T, \qquad \text{and} \qquad h_\sigma = \diam(\sigma),\quad \sigma\in\E,
\end{equation}
where $\diam(E)=\sup\{|x-y|,\, (x,y)\in E^2\}$. Finally, the size of the mesh is defined by
\begin{equation}\label{def diameter}
    h = \max_{K\in\T} h_K.
\end{equation}

\subsubsection{Regularity assumptions} In the sequel, we will use the following regularity assumptions on a mesh $\M$:
\smallskip

\begin{enumerate}[label=(HR.{\alph*})]
    \item \label{HR a volumes} There exists $\mu \in (0,\frac12]$ such that for all $K\in\T$, $K$ is star-shaped with respect to the open ball $B(x_K,\mu h)$. 
    \item \label{HR b faces} There exists $\lambda\in (0,\frac12]$ such that for all $\sigma\in\E$ we have $ \lambda^{d-1} h^{d-1} \Bdm \le |\sigma| $.\smallskip
\end{enumerate}
The notations and assumptions introduced in this section are illustrated on Figure~\ref{fig:mesh}.

\begin{figure}
\includegraphics[width = \textwidth]{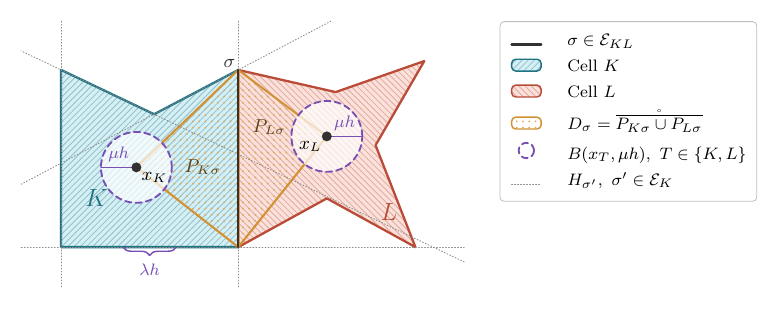}
\caption{Illustration of the notation and regularity assumptions on the mesh.}\label{fig:mesh}
\end{figure}

The constant $\mu$ in~\ref{HR a volumes} is sometimes referred to as chunkiness parameter and \ref{HR a volumes} is known as a quasi-uniformity regularity assumption \cite{fem_book}. These regularity assumptions have several consequences on the primal mesh as well as on its dual. More precisely, let us introduce $P_{K\sigma}$ as the $d$-dimensional (open) pyramid with base $\sigma$ and vertex $x_K$. A first consequence of \ref{HR a volumes} is that 
\begin{equation}\label{eq.partition}
P_{K\sigma}\cap P_{K \sigma'} = \emptyset, \hspace{2mm} \sigma \neq \sigma' \in\E_K, \qquad \text{and} \qquad \overline{K} = \bigcup_{\sigma\in\E_K} \overline{ P_{K\sigma}}.    
\end{equation}
Then, we can introduce the dual diamond cells as
\begin{equation*}
    \D = \{ D_\sigma, \, \sigma \in \E \}, \qquad \text{where} \qquad D_\sigma = \begin{cases} \mathring{\overline{ P_{K\sigma} \cup P_{L \sigma}}} \quad &{\mbox{if }\,\sigma \in \Eint, \, \mbox{ with } \,\sigma\in\EKL,}\\ 
    P_{K\sigma} \quad &{\text{if } \sigma \in \Eext, \, \mbox{ with } \sigma\in \E_K.}
    \end{cases}
\end{equation*}
An additional consequence of the regularity assumptions~\ref{HR a volumes}--\ref{HR b faces} is that all useful geometrical objects of given dimension in the mesh are comparable in terms of measure: a precise statement is given in Lemma~\ref{lem:regularity}. We furthermore discuss possible weakening of these regularity assumptions in Section~\ref{sec:discussion_regularity}.

\subsubsection{Discrete gradient and Sobolev norms} 

Given scalar values on cells $u=(u_K)_{K\in\T}\in\R^\T$, we define its associated piecewise constant function on the primal mesh $\M$
\begin{equation}\label{eq:piece_const_primal}
u^\M=\sum_{K\in\T} u_K \, \mathbf{1}_K,
\end{equation}
where $\mathbf{1}_E$ denotes the characteristic function of the set $E$. We additionally introduce an associated discrete gradient defined on the dual mesh $\D$ by
\begin{equation}\label{def discrete gradient}
\nabla^\M u = \sum \limits_{\substack{\sigma\in\Eint\\ \sigma\in\EKL}} |\sigma|\frac{u_L - u_K}{|D_\sigma|} \, \nKs \, \mathbf{1}_{D_\sigma}.
\end{equation}
Note that in the definition above, as well as in the rest of the paper, the summation holds over all $\sigma \in \Eint$. For each $\sigma\in \Eint$, the couple $(K,L)$ such that $\sigma\in \EKL$ is uniquely defined. This definition, which is classically used in the finite volume literature, see for instance~\cite[Equation~(4.11)]{CCLP03}, is discussed further in Appendix~\ref{sec:appendix}.

We will use the following Lebesgue and discrete Sobolev norms in order to establish several discrete functional inequalities. For $u\in \R^\T$, and $p\in[1,\infty]$, we define the Lebesgue norms
\begin{equation}\label{eq:lebesgue}
\|u\|_{0,p} = \|u^\M\|_{L^p(\Omega)},
\end{equation}
as well as the discrete Sobolev seminorms
\begin{equation}\label{eq:SobolevSeminorm}
|u|_{1,p} = \|\nabla^\M u\|_{L^p(\Omega)} = \left\{\begin{aligned}&\left(\sum\limits_{\substack{\sigma\in\Eint\\ \sigma\in\EKL}}|\sigma|^p|D_\sigma|^{1-p}|u_L-u_K|^p\right)^{1/p}&& \text{for } 1\leq p< \infty, \\
&\max_{\substack{\sigma\in\Eint\\ \sigma\in\EKL}}|\sigma|\frac{|u_L-u_K|}{|D_\sigma|}&& \text{for } p=\infty,\end{aligned}\right.
\end{equation}
and the discrete Sobolev norms, for $1\leq p\leq \infty$, 
\begin{equation}\label{es:Sobolev}
\qquad \|u\|_{1,p} = \|u\|_{0,p} + |u|_{1,p}.
\end{equation}

Some variants of the embeddings presented here will hold under homogeneous Dirichlet boundary conditions. We therefore need to adapt the definition of the seminorms in this context. For $\Gamma_0\subset \partial \Omega$ a portion of boundary such that $\Gamma_0=\cup_{\sigma \in \Eextz} \sigma$ for some $\Eextz\subset \Eext$, we define the discrete Sobolev seminorms associated to the homogeneous Dirichlet boundary condition on $\Gamma_0$ as
\begin{equation}\label{eq:SobolevSeminormDirichlet}
|u|_{1,p,\Gamma_0} = \left\{\begin{aligned}& \left(\sum\limits_{\substack{\sigma\in\Eint\\ \sigma \in\EKL}}|\sigma|^p|D_\sigma|^{1-p}|u_L-u_K|^p + \sum\limits_{\substack{\sigma\in\Eextz\\ \sigma\in \E_K}}|\sigma|^p|D_\sigma|^{1-p}|u_K|^p\right)^{1/p}&& \text{for } 1\leq p< \infty, \\
&\max\left(\max_{\substack{\sigma\in\Eint\\ \sigma\in\EKL}}|\sigma|\frac{|u_L-u_K|}{|D_\sigma|} , \max_{\substack{\sigma\in\Eextz\\ \sigma\in \E_K}}|\sigma|\frac{|u_K|}{|D_\sigma|}\right)&& \text{for } p=\infty.\end{aligned}\right.
\end{equation}
With this definition, the seminorm includes jumps at the boundary $\Gamma_0$, in the same spirit as in~\cite[Definition~2.1]{BCCHF}.

Note that the definition of $|\cdot|_{1,p}$ allows
one to apply elementary inequalities such as H\"older inequality, which directly gives
\begin{equation}\label{eq:Holder}
\|v\|_{0,s}\leq |\Omega|^{\frac1s-\frac1t}\|v\|_{0,t}\,,\qquad 
|v|_{1,s}\leq |\Omega|^{\frac1s-\frac1t}|v|_{1,t}\,,\qquad 1\leq s\leq t\leq\infty,
\end{equation}
and similarly for $|\cdot|_{1,p,\Gamma_0}$.

The discrete Sobolev seminorm is sometimes defined  alternatively without introducing a discrete gradient, as detailed in Appendix~\ref{sec:appendix}: the two seminorms are however equivalent for meshes satisfying the assumption~\ref{HR a volumes}, see Lemma~\ref{lem:equivseminorm}.

\subsection{Main results}
Our first result concerns discrete Sobolev embeddings with explicit improved asymptotics for the constants. A graphical representation of the cases and asymptotics of the theorem is given on Figure~\ref{fig:SobolevDiagram}.

\begin{theorem}[Discrete Sobolev embedding]\label{thm:Sobembedding}
 Let $\M$ be a mesh in the sense of {\normalfont \ref{HM a partition}--\ref{HM e cell centers}} satisfying the regularity assumptions {\normalfont \ref{HR a volumes}--\ref{HR b faces}}. Suppose that one of the following conditions holds. 
 \begin{itemize}
 \item Case A: $p\in(d,\infty]$, $q\in[p,\infty]$;
 \item Case B: $p = d$, $q\in[p,\infty)$;
 \item Case C: $p\in[1,d)$, $q\in[p,p^*]$ with $p^* = dp/(d-p)$.
 \end{itemize}
 Then, for all $u\in\R^\T$, the inequality 
\begin{equation}\label{eq:Sobemb}
 \|u\|_{0,q}\leq C^{\rm Sob}_{p,q}\|u\|_{1,p},
 \end{equation}
is satisfied for some constant $C^{\rm Sob}_{p,q}$ depending only on $p$, $q$, $\Omega$, and $\lambda$.
Furthermore, the following asymptotics hold, with constants depending only on $\Omega$ and $\lambda$,
\begin{equation}\label{eq:asymp}
\begin{aligned}
    C^{\rm Sob}_{p,\infty} &= O_{p\to d^+}\left((p-d)^{-\left(1-1/d\right)}\right), \\
    C^{\rm Sob}_{d,q} &= O_{q\to \infty}\left(q^{1-1/d}\right), \\
    C^{\rm Sob}_{p,p^*} &= O_{p\to d^-}\left((d-p)^{-\left(1-1/d\right)}\right).
\end{aligned}
\end{equation}
Finally, the inequality \eqref{eq:Sobemb} is fulfilled trivially with $C^{\rm Sob}_{p,q}=|\Omega|^{1/q-1/p}$ if $1\leq q< p$. 
\end{theorem}
\begin{figure}[!ht]
\centering
\includegraphics[width=.8\textwidth]{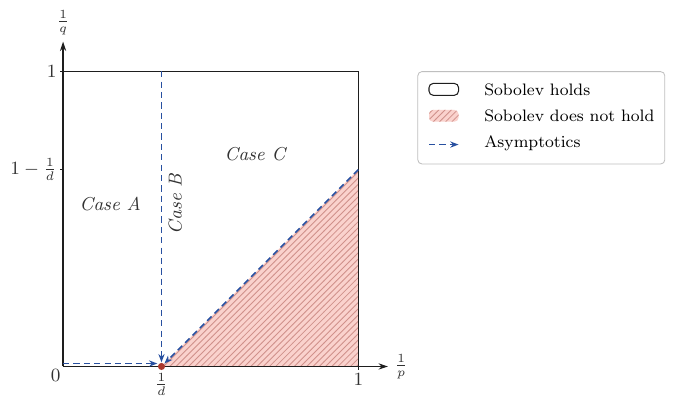}
\caption{Illustration of the cases in Theorem~\ref{thm:Sobembedding}.} \label{fig:SobolevDiagram}
\end{figure}

Theorem~\ref{thm:Sobembedding} is a consequence of a more precise result stated in Proposition~\ref{prop:preciseembedding}. Its proof relies on a discrete potential inequality of the form~\eqref{Ineg: discrete potential} established in Proposition~\ref{prop.fundineq}, which is inspired by the strategy of~\cite{Trudinger67}.

Then, making further use of the refined inequality of Proposition~\ref{prop:preciseembedding}, we prove in Sections~\ref{sec: proofs Thm Poincare} and~\ref{sec: Gagliardo--Nirenberg proof} the following three results which are discrete counterparts of Poincar\'e-Sobolev, and Gagliardo--Nirenberg--Sobolev inequalities. For the latter, we recover the whole range of expected parameters compared to the continuous case. In particular the super-critical case $p>d$ and $q=\infty$ was not available in the literature. This case requires a refined analysis through a localization of our estimates, which is carried out in Section~\ref{sec: Gagliardo--Nirenberg proof}.

\begin{theorem}[Discrete Poincar\'e-Sobolev inequality for zero mean functions]\label{thm:PoincarSob}
Let the assumptions of Theorem~\ref{thm:Sobembedding} hold. Then, there exists a constant $C^{0}_{p,q}>0$ depending only on $p$, $q$, $\Omega$ and $\lambda$ such that, for all $u\in\R^\T$,
\begin{equation}\label{eq:PoincSobWir}
\|u - \bar{u}\|_{0,q}\leq C^{0}_{p,q}|u|_{1,p},\qquad \bar{u} = \frac{1}{|\Omega|}\sum_{K\in\T} |K|u_K. 
\end{equation}
Moreover, the constant $C^{0}_{p,q}$ behaves asymptotically as in~\eqref{eq:asymp}.
\end{theorem}

\begin{theorem}[Discrete Poincar\'e-Sobolev inequality with Dirichlet boundary conditions]\label{thm:PoincarSobDir}
Let $\Gamma_0\subset \partial \Omega$ be a non-empty portion of the boundary. Then, under the assumptions of Theorem~\ref{thm:Sobembedding}, and assuming that $\Gamma_0=\cup_{\sigma \in \Eextz} \sigma$ for some $\Eextz\subset \Eext$, there exists a constant $C^{\Gamma_0}_{p,q}>0$ depending only on $p$, $q$, $\Omega$, $\Gamma_0$, and $\lambda$ such that, for all $u\in\R^\T$,
\begin{equation}\label{eq:PoincSob}
\|u\|_{0,q}\leq C^{\Gamma_0}_{p,q}|u|_{1,p,\Gamma_0}. 
\end{equation}
Furthermore, the constant $C^{\Gamma_0}_{p,q}$ behaves asymptotically as in~\eqref{eq:asymp}.
\end{theorem}

\begin{theorem}[Discrete Gagliardo--Nirenberg--Sobolev inequality]\label{thm:GagliardoNirenbergGeneral}
Let $\M$ be a mesh in the sense of {\normalfont \ref{HM a partition}--\ref{HM e cell centers}} satisfying the regularity assumptions {\normalfont \ref{HR a volumes}--\ref{HR b faces}}. Let $1\leq p,q, r\leq \infty$ and $\theta \in[0,1]$ such that
\begin{align*}
    \frac1q = \theta\left(\frac{1}{p}-\frac1d\right) + (1-\theta)\frac{1}{r}.
\end{align*}
Then, provided $\theta \neq 1$ when $p=d$, there exists a constant $C>0$ which depends only on $p$, $q$, $r$,  $\lambda$, $\mu$ and $\Omega$ such that for any $u\in\R^\T$,
\begin{align}\label{ineg GNS}
    \|u\|_{0,q} \leq C \, \|u\|_{1,p}^{\theta} \, \|u\|_{0,r}^{1-\theta}.
\end{align}
Similarly, estimate~\eqref{ineg GNS} holds, up to changing the constant, with $\|u\|_{1,p}$ replaced by $|u|_{1,p}$ if $u$ has zero mean value, i.e., $\bar{u}=0$. Finally, it also holds, up to changing the constant, if $\Gamma_0 \subset \partial \Omega$ as in Theorem~\ref{thm:PoincarSobDir}, and $\|u\|_{1,p}$ is replaced by $|u|_{1,p,\Gamma_0}$.
\end{theorem}
Though we do not report it, the dependence of the discrete Gagliardo--Nirenberg--Sobolev constants on $p$, $q$, and $r$ can be tracked explicitly in the proof of Theorem~\ref{thm:GagliardoNirenbergGeneral}. 

Finally, as a direct application of the discrete Sobolev embeddings with the asymptotics \eqref{eq:asymp} for the constants, we prove the following discrete version of the Trudinger inequality.

\begin{theorem}[Discrete Trudinger inequality]\label{thm:Trudinger1}
Let $\M$ be a mesh in the sense of {\normalfont \ref{HM a partition}--\ref{HM e cell centers}} satisfying the regularity assumptions {\normalfont \ref{HR a volumes}--\ref{HR b faces}}. Then, there exists an explicit constant $\alpha_\star>0$ depending only on  $\Omega$ and $\lambda$ such that for all $0\leq\alpha<\alpha_\star$, there is $C_\alpha>0$ which depends only on  $\Omega$, $\lambda$, and $\alpha$ such that
\begin{equation}\label{eq:MT1}
\sup_{\substack{u\in\R^\T\\\|u\|_{1,d}\leq 1}} \sum_{K\in\T}|K|\exp\left(\alpha |u_K|^{\frac{d}{d-1}}\right)\leq C_\alpha.
\end{equation}
\end{theorem}
To prove this result, we adapt at the discrete level the proof of~\cite{Trudinger67} in Section~\ref{sec: Trudinger}. In the same section, we also show an alternative form of Trudinger inequality as well as versions with zero mean value  and Dirichlet boundary conditions (see Proposition~\ref{prop:Trudinger2}, Proposition~\ref{prop:Trudinger3} and Proposition~\ref{prop:Trudinger4}).

\subsection{Discussion on the regularity assumptions}\label{sec:discussion_regularity}
In the following, we mention two ways to weaken our quasi-uniform regularity assumption~\ref{HR a volumes} and assumption~\ref{HR b faces}, which allow for local refinements of the mesh.

First, notice that in Theorems~\ref{thm:Sobembedding} to \ref{thm:Trudinger1}, with the exception of Theorem~\ref{thm:GagliardoNirenbergGeneral} (case $p\geq d$), the constants depend on $\lambda$, but not on $\mu$. Therefore, the two regularity parameters in~\ref{HR a volumes} and~\ref{HR b faces} do not play a similar role. Indeed, considering a sequence of meshes $\M_n$ with size $h_n\to 0$, in order to apply our main results uniformly in $n$ we require~\ref{HR b faces} for a fixed $\lambda>0$ while for~\ref{HR a volumes} we only require a sequence $\mu_n>0$, possibly vanishing when $n\to \infty$. As explained before, for a fixed mesh, assumption~\ref{HR a volumes} with some $\mu>0$ makes it possible to define the dual mesh and the discrete gradient. In fact, those constructions, and the aforementioned theorems, hold if we replace~\ref{HR a volumes} by
 \begin{enumerate}[label=(HR.{\alph*}')]
    \item \label{HR a prime} For all $K\in\T$, $K$ is star-shaped with respect to $x_K$ and no diamond cell is empty.\smallskip
 \end{enumerate}

Second, except for the proof of Theorem~\ref{thm:GagliardoNirenbergGeneral} (case $p\geq d$), we could also slightly weaken the global assumption~\ref{HR b faces} to the following more local assumption:
\begin{enumerate}[label=(HR.{\alph*}')]
    \setcounter{enumi}{1}
    \item \label{HR b prime} There exists $\lambda\in (0,\frac12]$ such that for all $K\in\T$, $\sigma\in\E_K$ we have $ \lambda^{d-1} h_K^{d-1} \Bdm \le |\sigma| $.\smallskip
\end{enumerate}
The proofs require some adaptation, in that all the estimates leading to the crucial potential estimate of Proposition~\ref{prop.fundineq} should be ``de-centered'' from $x_K$ to any arbitrary $x$. In particular, for $x\in K$ and $\sigma\in\Eint$ the volume of the shadow of the face $\sigma$ emanating from $x$, see \eqref{def.AKsigma}, can be estimated with local sizes around $\sigma$, independently of $h_K$. 

Finally, the classical local regularity hypotheses for discontinuous Galerkin methods (see \cite[Lemma~1.61 and the preceding discussions]{dg_book}) are shape- and contact-regularity of a submesh and uniform star-shapedness of all cells $K$ with respect to a ball  of radius uniformly comparable to $h_K$. They are stronger than~\ref{HR a prime} and~\ref{HR b prime}, so that the aforementioned theorems also hold in this setting. 
The uniform star-shapedness of all cells $K$ with respect to a ball, and not merely  \ref{HR a prime}, also plays a critical role in the proof of our main result for broken Sobolev spaces Theorem~\ref{thm:DG}, where functions are not constant anymore inside each cell.

\section{Discrete potential inequality}\label{sec: discrete potential inequality}

The purpose of this section is to prove the following inequality which constitutes one of the core contributions of the present paper.
\begin{proposition}\label{prop.fundineq}
    Let $\M$ be a mesh in the sense of {\normalfont \ref{HM a partition}--\ref{HM e cell centers}} satisfying the regularity assumptions {\normalfont \ref{HR a volumes}--\ref{HR b faces}}. Then, there exists an explicit constant $C=C(\Omega,\lambda)$ such that, for all $u\in\R^\T$,
    \begin{equation*}
        | u^\M |(x) \le C \left( \|u\|_{0,1} + \int_\Omega \frac{|\nabla^\M u|(\xi)}{|x-\xi|^{d-1}} \, \dd \xi \right), \qquad \mbox{for a.e. }x\in \Omega.
    \end{equation*}
\end{proposition}

The above inequality is reminiscent of the potential estimate developed in the continuous case which is at the heart of the original proof of Trudinger \cite{Trudinger67}. In order to prove this inequality, we first define ``discrete paths'' joining two points $x$ and $y$ such that $y\in k_\Omega(x)$. In practice, such a discrete path is a sequence of faces connecting iteratively neighboring cells that plays at the discrete level the role of the segment $[x,y]$, see Figure~\ref{fig:path}. A discrete integration along the path gives, after (standard) integration for $y\in k_\Omega(x_K)$ (see Lemma~\ref{lem:ineqAksig}),
\begin{equation*}
    |k_\Omega| |u_K| \le \int_{k_\Omega (x_K)} |u^\M| \dd x+ \sum_{\substack{\sigma\in\Eint\\ \sigma\in\EKpLp}} A_{K\sigma} \frac{|\sigma| \, |u_{K'}-u_{L'}|}{|D_\sigma|} |D_\sigma|,
\end{equation*}
where for any $K\in\T$ and $\sigma\in \Eint$,
\begin{equation}\label{def.AKsigma}
    A_{K\sigma} = \frac{|S_{K\sigma}|}{|\sigma|},\quad
\text{with}\quad
    S_{K\sigma} = \{ x_K + t (s-x_K), \, s\in \overline{\sigma}, \, t\ge 1 \}\cap \Omega,
\end{equation}
and we call $S_{K\sigma}$  
 the shadow of the face $\sigma$ emanating from $x_K$, see Figure~\ref{fig:conesecurity}. To estimate the geometric quantity $A_{K\sigma}$, we proceed in two steps: we first estimate the volume of $S_{K\sigma}$, providing a first estimate for $A_{K\sigma}$ that depends on the size of the mesh $\M$ (Lemma~\ref{lem.Aksigma}). Then, we use the regularity assumptions on the mesh to show that, for $x\in K$ and $\xi\in D_\sigma$, $A_{K\sigma} \lesssim |x-\xi|^{-(d-1)}$ independently of the mesh size (Lemma~\ref{lem.Aksigma_reg}). It then suffices to reinsert this estimate in the inequality of Lemma~\ref{lem:ineqAksig} to prove Proposition~\ref{prop.fundineq}.

We start by deriving some geometrical results on the mesh.

\subsection{Discrete paths} For any $x$, $y\in\Omega$, we define
\begin{equation*}
    C(x,y):=\{\sigma \in \E:\, \overline{\sigma}\cap [x,y] \neq \emptyset\} .
\end{equation*}
Note that if $y$ belongs to the cone $k_\Omega(x)$, then the cone condition implies that the segment $[x,y]$ is included in $\Omega$, which implies that $C(x,y)\subset \Eint$. Furthermore, the following lemma holds.

\begin{lemma}\label{lem.triangineqalongpath}
    Let $\M$ be a mesh in the sense of {\normalfont \ref{HM a partition}--\ref{HM e cell centers}}. Then, for all $K \in \T$, $x \in K$, and for almost every $y\in k_\Omega(x)$, we have
    \begin{equation}\label{ineg path}
        |u^\M|(x) \le |u^\M|(y) + \sum_{\substack{\sigma \in C(x,y)\\ \sigma\in\EKpLp}} |u_{K'}-u_{L'}|.
    \end{equation}
\end{lemma}

\begin{figure}
\includegraphics[width = \textwidth]{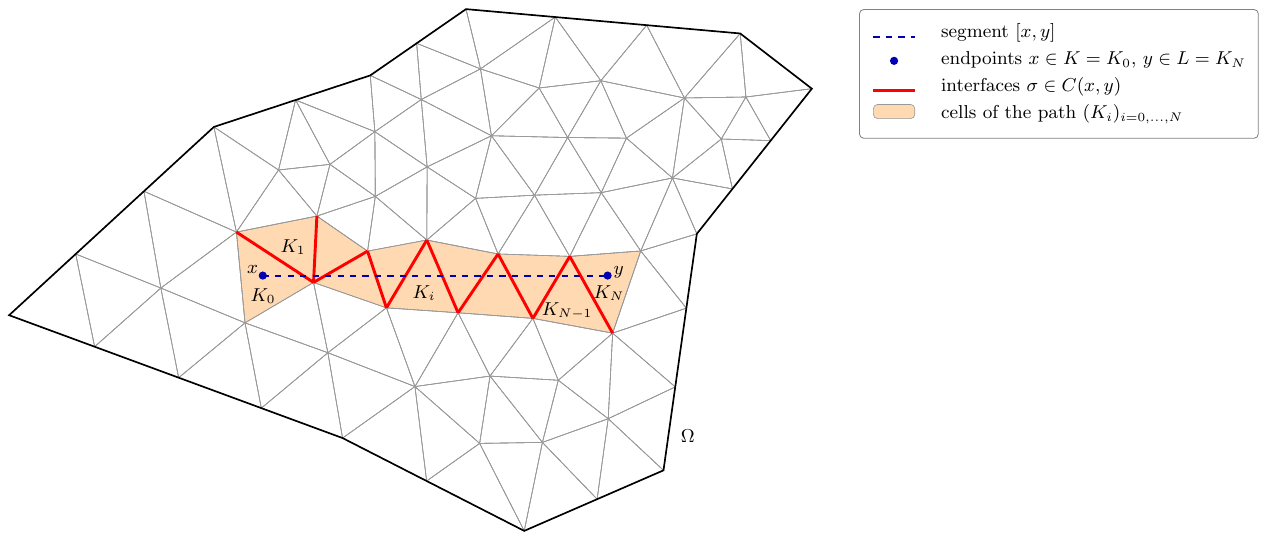}
\caption{Illustration of a discrete path between two points in a triangular mesh of the domain $\Omega$. A path $K_0,\dots,K_N$ is well defined for almost every $x,y\in\Omega$ with $y\in k_\Omega(x)$ (see proof of Lemma~\ref{lem.triangineqalongpath}).}\label{fig:path}
\end{figure}

\begin{proof}
    Let $x\in K$ and $y\in k_\Omega(x)$, with $y\in L$. The idea is to define a path in the mesh following the segment $[x,y]$, that is, a finite sequence of faces $(\sigma_i)_i\subset \Eint$, distinct, spanning $C(x,y)$ and ordered in the sense that there is a finite sequence of iteratively neighboring cells $(K_i)_i$ such that $\sigma_i\in \EKi$ for all $i$. To avoid ambiguity in the definition of the path we shall first discard the cases where $(x,y)$ intersects the boundary of the faces of the cells by showing that it almost never occurs. More precisely, let us define the set
\begin{equation*}
    \cN_x = \bigcup_{\sigma \in \Eint}  \{ x + t (z-x) \, : \, t \in \R,\,  z \in \partial\sigma\}.
\end{equation*}
Notice that $\cN_x$ is included in a finite union of $(d-1)$-dimensional affine subspaces and is thus negligible for the $d$-dimensional Lebesgue measure. Thus, for almost every $y\in k_\Omega(x)$ we have that $y\notin \cN_x$, that is, $(x,y)$ does not intersect the boundary of any cell face. We can now define a precise path.  For almost every $y\in k_\Omega(x)$, the segment $[x,y]$ intersects all elements of $C(x,y)$ which we can order in some sequence $(\sigma_i)_{i\in\{1,\dots,N(x,y)\}}$ such that $\sigma_i \in \EKi$ with $K_0 := K$ and $K_{N(x,y)} := L$. Note that the $(K_i)_i$ need not be distinct and $N(x,y)$ may be equal to $0$. With such a path, one can write
    \begin{align}
        |u_K| = \left| u_L + \sum_{i=1}^{N(x,y)} (u_{K_{i-1}}-  u_{K_{i}} )\right|&\le& |u_L| + \sum_{i=1}^{N(x,y)} |u_{K_{i-1}}-  u_{K_{i}}|\nonumber \\
        &=& |u_L| + \sum_{\substack{\sigma \in C(x,y)\\ \sigma\in\EKpLp}} |u_{K'}-u_{L'}|.\nonumber
    \end{align}
This concludes the proof of Lemma~\ref{lem.triangineqalongpath}.
\end{proof}
As a consequence, we have the following pointwise estimate.
\begin{lemma}\label{lem:ineqAksig}
Let $\M$ be a mesh in the sense of {\normalfont \ref{HM a partition}--\ref{HM e cell centers}}.  Then, for all $K\in\T$, $x\in K$,
\begin{equation*}
     |u^\M|(x) \le \frac{1}{|k_\Omega|}\int_{k_\Omega (x_K)} |u^\M|\, \dd y + \frac{1}{|k_\Omega|}\sum_{\substack{\sigma\in\Eint\\ \sigma\in\EKpLp}} A_{K\sigma} \frac{|\sigma||u_{K'}-u_{L'}|}{|D_\sigma|} |D_\sigma|,
\end{equation*}
where the $A_{K\sigma}$ are defined in equation~\eqref{def.AKsigma}.
\end{lemma}
\begin{proof}
Let $K\in\T$. From the inequality~\eqref{ineg path} of Lemma~\ref{lem.triangineqalongpath} with $x = x_K$, then integrating over $y\in k_\Omega(x_K)$, one has
\begin{align*}
    |k_\Omega| |u_K| &\le \int_{k_\Omega (x_K)} |u^\M|\, \dd y + \int_{k_\Omega (x_K)}\sum_{\substack{\sigma \in C(x_K,y)\\ \sigma\in\EKpLp}} |u_{K'}-u_{L'}|\, \dd y \\
    &= \int_{k_\Omega (x_K)} |u^\M|\, \dd y + \sum_{\substack{\sigma\in\Eint\\ \sigma\in\EKpLp}}|S_{K\sigma} \cap k_\Omega(x_K)| |u_{K'}-u_{L'}|\\
    & \le \int_{k_\Omega(x_K)} \left|u^{\M}\right| \, \dd y + \sum_{\substack{\sigma\in\Eint\\ \sigma\in\EKpLp}}|S_{K\sigma}| |u_{K'}-u_{L'}|,
\end{align*}
which leads to the desired inequality.
\end{proof}

\subsection{Estimate of the shadow of a face}\label{sec:constructions}

In all this section, $K\in\T$ and $\sigma\in \Eint$ are fixed. The goal of this section is to evaluate $A_{K\sigma}$ defined in \eqref{def.AKsigma}. For this purpose, let us first define some notation, see also Figure~\ref{fig:conesecurity}. First, we define,
\[
\widehat{\sigma} = \overline{\mathrm{Conv(\sigma)}},
\]
the closure of the convex hull of $\sigma$. In all this section we suppose that $\dd(x_K,\widehat{\sigma})>0$.
\begin{remark} 
    In dimension $d\ge3$, the assumption that $\dd(x_K,\widehat{\sigma})>0$ is not always satisfied for general regular meshes: see Appendix~\ref{sec:appendix} for an example with $x_K \in \widehat{\sigma}$. This assumption is critical for the constructions in this section but we will see in the following section (Lemma~\ref{lem.Aksigma_reg}) that when $x_K \in \widehat{\sigma}$, $A_{K\sigma}$ can be estimated in a fairly direct way.
\end{remark}
Note that $h_\sigma = \diam(\widehat{\sigma})$. By convexity and compactness of $\widehat{\sigma}$ we can define,
\begin{equation*}
    x_{K{\sigma}} = \underset{y\in \widehat{\sigma} } \argmin \, |x_K-y| \in \widehat{\sigma}.
\end{equation*}

\begin{figure}
\includegraphics[width = \textwidth]{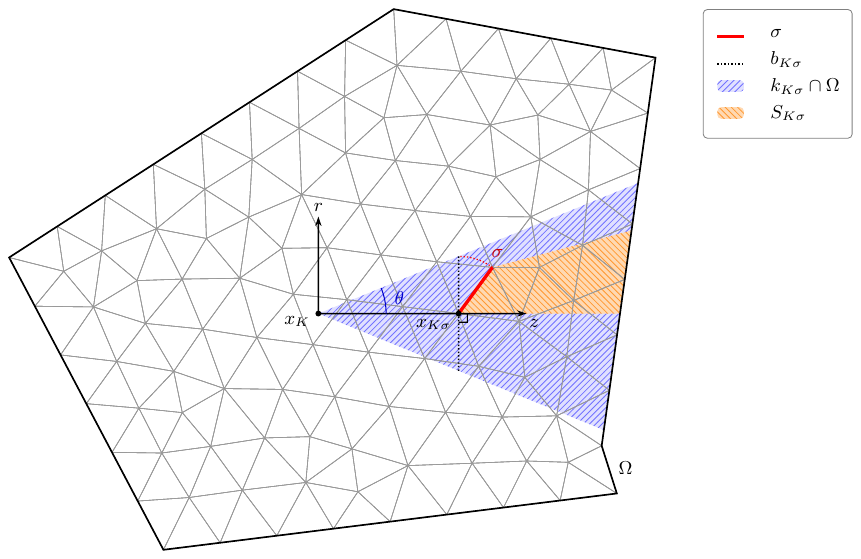}
\caption{Illustration for $d=2$ of the geometric objects defining the security cone $k_{K\sigma}$ used to estimate the volume of the shadow set $S_{K\sigma}$.\label{fig:conesecurity}}
\end{figure}

When $d\ge3$, we write $(z,r,\phi_1,\dots,\phi_{d-2})\in \R\times\R_+\times[0,\pi]^{d-3}\times[0,2\pi]$ the hypercylindrical coordinates in $\R^d$, centered in $x_K$, where $z$ is the coordinate along the $(x_K,x_{K {\sigma}})$ line and $(r,\phi_1,\dots,\phi_{d-2})$ are the $d-1$ spherical coordinates in the hyperplane orthogonal to $(x_K,x_{K {\sigma}})$. In order to unify notations, when $d=2$ we use the ``cylindrical'' coordinates $(z,r,\phi_1)\in \R\times\R_+\times\{\pm 1\}$, where $z$ is the coordinate along the $(x_K,x_{K\sigma})$ line and $r\phi_1$ the coordinate along the orthogonal axis. We then use the system of coordinates $(z,r,\phi_1,\dots,\phi_{d'})$ where $d'=1$ if $d=2$, and $d'=d-2$ if $d\ge3$. In this system of coordinates, by definition of $x_{K\sigma}$ we have $x_{K\sigma} = (z_{K\sigma},0,0, \dots, 0)$ with $z_{K\sigma}=\dd(x_K,\widehat{\sigma})$.
We now define the angle
\begin{equation*}
 \theta = \arctan \frac{ h_\sigma}{\dd(x_K,\widehat{\sigma})}\in \left(0,\frac\pi2\right),
\end{equation*}
the $(d-1)$-dimensional ball,
\[
b_{K\sigma} = \{y\in\mathbb{R}^d:\, z=z_{K\sigma} \text{ and } r< h_\sigma \},
\]
and finally the cone,
\begin{align*}
    k_{K\sigma} & = \{x_K + t(b-x_K):\  t>0\ \text{and}\  b\in b_{K\sigma}\}\\[.3em]
    & = \{(z,r,\phi_1,\dots,\phi_{d'}): \  z>0 \ \text{and}\  r < z \tan \theta\}.
\end{align*}

Endowed with these notations, we can evaluate $A_{K\sigma}$ as follows.

\begin{lemma}\label{lem.Aksigma}
Let $\M$ be a mesh in the sense of {\normalfont \ref{HM a partition}--\ref{HM e cell centers}}. Then, there exists an explicit constant $C_0>0$ that depends only on $d$, such that for any $K\in\T$ and $\sigma\in \Eint$ with $x_K\notin \widehat{\sigma}$ we have
\begin{equation*}
      A_{K\sigma} \le C_0\frac{\diam(\Omega)^d}{|\sigma|}\,\left(\dfrac{ h_\sigma}{\dd(x_K,\widehat{\sigma})}\right)^{d-1}.
\end{equation*}
\end{lemma}
\begin{proof}
Observe that if $s = (z_s, r_s,\dots)\in\overline{\sigma}$ then, since $x_{K\sigma}$ is the projection of $x_K$ onto the closed and convex domain $\widehat{\sigma}$, one has $z_s\geq z_{K\sigma}$. Moreover, $|s-x_{K\sigma}|\leq \diam(\widehat{\sigma}) = h_\sigma$, hence $r_s\leq h_\sigma$. In particular, $s\in \overline{k_{K\sigma}}$, and thus $S_{K\sigma} \subset \overline{k_{K\sigma}}\cap\Omega$. Therefore, we can evaluate
    \begin{align*}
      |S_{K\sigma}| \le |k_{K\sigma}\cap\Omega|  &\leq  |\{z\in (0, \diam(\Omega)) \text{ and } r < z \tan \theta \}|\\ &= \int_{0}^{\diam(\Omega)} |B_{d-1}| \left(z \tan\theta\right)^{d-1} \dd z \\
      & =   \frac{|B_{d-1}|} d \diam(\Omega)^d (\tan\theta)^{d-1},
    \end{align*}
    which completes the proof of the desired estimate with $C_0 = d^{-1} |B_{d-1}|$.
\end{proof}

\subsection{Use of the regularity assumptions and proof of Proposition~\ref{prop.fundineq}}
Building on the previous section, we can now use the regularity of the mesh to estimate $A_{K\sigma}$ in terms of the continuous variables.
\begin{lemma}\label{lem.Aksigma_reg} Let $\M$ be a mesh in the sense of {\normalfont \ref{HM a partition}--\ref{HM e cell centers}} satisfying the regularity assumptions {\normalfont \ref{HR a volumes}--\ref{HR b faces}}.
Then, there exists an explicit constant $C_1>0$ that depends only on $d$, such that for any $K\in\T$ and $\sigma\in \Eint$ and for any $x\in K$ and a.e. $\xi\in D_\sigma$ we have
\begin{equation}\label{eq.Aksigma_reg_2}
      A_{K\sigma} \le \frac{C_1\diam(\Omega)^d}{\lambda^{d-1}|x-\xi|^{d-1}}.
\end{equation}
\end{lemma}

\begin{proof}We split the proof in two cases.\\
\noindent\emph{Case 1: $|\xi-x|\le 4 h$.}
Using first the definition of $A_{K\sigma}$ in~\eqref{def.AKsigma}, the regularity assumption {\normalfont \ref{HR b faces}}, and the isodiametric inequality, we estimate $A_{K\sigma}$ as
\begin{equation}
    A_{K\sigma} \le \frac{|\Omega|}{|\sigma|} \le \frac{|\Omega|}{\lambda^{d-1}h^{d-1}|B_{d-1}|} \le \frac{\diam(\Omega)^d \Bd}{2^d \lambda^{d-1}h^{d-1}\Bdm}\le \frac{2^{d-2} \diam(\Omega)^d \Bd}{\lambda^{d-1}\Bdm}\frac{1}{|x-\xi|^{d-1}},
\end{equation}
which proves the desired inequality when $|\xi-x|\le 4 h$.
\medskip

\noindent\emph{Case 2: $|\xi-x| > 4 h$.} Let us call $L$ a cell such that $\xi \in \overline{P_{L\sigma}}$.
First, observe that
\begin{equation*}
    |x-\xi| \le |x-x_K| + |x_K-x_{K\sigma}| + |x_{K\sigma}-\xi| \le  \dd(x_K,\widehat{\sigma}) + 2h \le \dd(x_K,\widehat{\sigma}) + \frac12 |x-\xi|,
\end{equation*}
where we used that $\overline{P_{L\sigma}}\cup \widehat{\sigma}$ is a subset of the convex hull of $\overline{L}$, whose diameter is equal to $\diam(L)\le h$, to control $|x_{K\sigma}-\xi|$. As a consequence,
\begin{equation}\label{eq.distfar}
\dd(x_K,\widehat{\sigma}) \ge \frac12 |x-\xi| >0.
\end{equation}
In particular, $x_K\notin\widehat{\sigma}$ and we can use the constructions of Section~\ref{sec:constructions}. From there, combining Lemma~\ref{lem.Aksigma} and the regularity assumption {\normalfont \ref{HR b faces}}, we get
\begin{equation*}
      A_{K\sigma} \le \frac{C_0 \diam(\Omega)^d}{\lambda^{d-1}h^{d-1}|B_{d-1}|}\,\left(\dfrac{ h}{\dd(x_K,\widehat{\sigma})}\right)^{d-1}\le \frac{2^{d-1} C_0 \diam(\Omega)^d}{\lambda^{d-1}|B_{d-1}|}\,\dfrac{ 1}{|x-\xi|^{d-1}},
\end{equation*}
which concludes the proof of the lemma with $C_1 = \max(2^{d-2}|B_d|/|B_{d-1}|, 2^{d-1}/d)$ .
\end{proof}
Now we have all the elements to prove Proposition~\ref{prop.fundineq}.

\begin{proof}[Proof of Proposition~\ref{prop.fundineq}]
Combining Lemma~\ref{lem:ineqAksig} and inequality \eqref{eq.Aksigma_reg_2}, one has 
\begin{align*}
    |u_K| &\le \frac{1}{|k_\Omega|} \int_{k_\Omega(x_K)} \left|u^{\M}\right| \, \dd y  + \frac{1}{|k_\Omega|} \sum_{\substack{\sigma\in\Eint\\ \sigma\in\EKpLp}} A_{K\sigma} \, |\sigma| \frac{|u_{K'}-u_{L'}|}{|D_\sigma|} |D_\sigma|\\
    &= \frac{1}{|k_\Omega|} \int_{k_\Omega(x_K)} \left|u^{\M}\right| \, \dd y  +  \frac{1}{|k_\Omega|} \int_\Omega \sum_{\sigma\in\Eint} A_{K\sigma} |\nabla^\M u|(\xi) {\bf 1}_{D_\sigma}(\xi) \dd \xi\\
    &\leq\frac{1}{|k_\Omega|} \int_{k_\Omega(x_K)} \left|u^{\M}\right| \, \dd y  + \frac{C_1\diam(\Omega)^d}{\lambda^{d-1} \, |k_\Omega|} \int_\Omega \sum_{\sigma\in\Eint}\frac{|\nabla^\M u|(\xi)}{|x-\xi|^{d-1}}{\bf 1}_{D_\sigma}(\xi) \dd \xi,
\end{align*}
for any $x\in K$, which leads to the  claim with $C = |k_\Omega|^{-1}\max(1, C_1\diam(\Omega)^d\lambda^{1-d})$.
\end{proof}

\section{Proofs of the main results}\label{sec: proofs main results}

\subsection{Sobolev embeddings: proof of Theorem~\ref{thm:Sobembedding}}\label{sec: proof Sob embeddings}

In the next proposition, we prove a more precise version of \eqref{eq:Sobemb}. 

\begin{proposition}\label{prop:preciseembedding}
 Let $\M$, $p$ and $q$ be as in Theorem~\ref{thm:Sobembedding}. Then, for all $u\in\R^\T$, the inequality 
\begin{equation}\label{eq:preciseemb1}
\|u\|_{0,q}\leq C\left(|\Omega|^{\frac{1}{q}}\|u\|_{0,1}  + \diam(\Omega)^{1+\frac{d}{q} -\frac{d}{p}}\Lambda_{p,q}|u|_{1,{p}}\right),
\end{equation}
holds with 
\begin{equation}\label{eq:constantSob}
\Lambda_{p,q} = \left\{\begin{aligned}
&\left(\frac{1+\frac1q - \frac1{p}}{\frac1d +\frac1q-\frac1{p}}\right)^{1+\frac1q-\frac1{p}}&&\text{if }\ \frac1d +\frac1q-\frac1{p}>0,\\[.5em]
&\left(\frac{p}{p-1}\right)^{1-\frac1d} + q^{1-\frac1d}&&\text{if }\ \frac1d +\frac1q-\frac1{p}=0\text{ and }p>1,\\
& \qquad\qquad1&&\text{if }\ \frac1d +\frac1q-\frac1{p}=0\text{ and }p=1.\\
\end{aligned}\right.
\end{equation}
and where $C$ depends only on $\Omega$ and $\lambda$.
\end{proposition}

\begin{remark}[Asymptotics of $\Lambda_{p,q}$] Let us underline that the constant $\Lambda_{p,q}$ is not meant to be optimal for the whole range of $p,q$ considered in Theorem~\ref{thm:Sobembedding}, and merely yields an upper bound. For instance the blow-up of $\Lambda_{p,p^*}$ as $p\to1$ is an artifact of the proof (based on the Hardy--Littlewood--Sobolev inequality for this part). It has however the optimal order of blow-up (that of the continuous constants \eqref{Constant Sobolev asymptotic}) near $(p^{-1},q^{-1}) = (d^{-1},0)$, which is the crucial part to derive the Trudinger inequality.
\end{remark}
\begin{proof}
Let $\M$ be a mesh in the sense of {\normalfont \ref{HM a partition}--\ref{HM e cell centers}} satisfying the regularity assumptions {\normalfont \ref{HR a volumes}--\ref{HR b faces}}. 

\smallskip

\noindent\emph{Step 1:} We start with Cases A, B, and C of Theorem~\ref{thm:Sobembedding}, except for $q = pd/(d-p)$ with $p<d$ of Case C which will be handled in Step 2. In turn we have  $1 {\le {p}\le q \le } \infty$ and 
\begin{equation}\label{hyp:Embed_dqr}
  \frac1{p}<\frac1q+\frac1d.
\end{equation}
Let us show that there exists an explicit constant $C>0$ which depends only on $\Omega$ and $\lambda$ such that for any $u\in\R^\T$ and any $f\in L^{q'}(\Omega)$ with ${q'}^{-1} + q^{-1} = 1$, one has
\begin{equation}\label{eq:sobdebase}
\left|\int_\Omega u^\M(x) f(x)\,\dd x\right|\leq C\left(\|u\|_{0,1}\|f\|_{L^1(\Omega)} + \diam(\Omega)^{1+\frac dq-\frac d{p}} \Lambda_{p,q}|u|_{1,{p}}\|f\|_{L^{q'}(\Omega)}\right).
\end{equation}
From Proposition~\ref{prop.fundineq}, for any test function $f\in L^{q'}(\Omega)$, one has
\[
\left|\int_\Omega u^\M(x) f(x)\,\dd x\right|\leq  C\left(\|u\|_{0,1}\|f\|_{L^1(\Omega)} + \iint_{\Omega\times\Omega} \frac{|\nabla^\M u(\xi)f(x)|}{|x-\xi|^{d-1}}\,\dd x\,\dd \xi\right).
\]
Now, define $G = |\nabla^\M u|\mathbf{1}_\Omega$ and $w  = |\cdot|^{1-d}\mathbf{1}_{|\cdot|\leq\diam(\Omega)}$. Then, by H\"older's and Young's convolution inequalities
\begin{align*}
\iint_{\Omega\times\Omega} \frac{|\nabla^\M u(\xi)f(x)|}{|x-\xi|^{d-1}}\dd x\dd \xi&\leq \|f\|_{L^{q'}(\Omega)} \|G\ast w\|_{L^{q}(\R^d)}\\
&\leq \|f\|_{L^{q'}(\Omega)} \|G\|_{L^{p}(\R^d)} \|w\|_{L^{s}(\R^d)},
\end{align*}
with $s^{-1} = 1+q^{-1}-p^{-1}$ and, by assumption, $s\in[1, d/(d-1))$. Thus, a direct computation leads to
\begin{align*}
\|w\|_{L^{s}(\R^d)} = \left(\int_{B(0,\diam(\Omega))} \frac{\dd z}{|z|^{(d-1)s}}\right)^{1/s} &= \diam(\Omega)^{1-d+\frac ds} \Bd^{\frac1s} \Lambda_{p,q}.
\end{align*}
 Since $\|G\|_{L^{p}(\R^d)} = |u|_{1,p}$ this yields
\[
\iint_{\Omega\times\Omega} \frac{|\nabla^\M u(\xi)f(x)|}{|x-\xi|^{d-1}}\,\dd x\,\dd \xi\leq C_d \diam(\Omega)^{1 +\frac{d}q-\frac{d}p } \Lambda_{p,q}|u|_{1,p}\|f\|_{L^{q'}(\Omega)},
\]
for some constant $C_d$ depending only on $d$ and $\Lambda_{p,q}$ defined by~\eqref{eq:constantSob}.
This shows \eqref{eq:sobdebase} under condition~\eqref{hyp:Embed_dqr}. If $q<\infty$ then \eqref{eq:sobdebase} implies
\[
\|u\|_{0,q}^q\leq C\left( \|u\|_{0,1}\|u\|_{0,q-1}^{q-1} + \diam(\Omega)^{1 +\frac{d}q-\frac{d}p } \Lambda_{p,q}|u|_{1,{p}}\|u\|_{0,q}^{q-1}\right),
\]
which yields \eqref{eq:preciseemb1} by H\"older's inequality.
In the case $q=\infty$ and $p>d$ of Case A, take $f = \mathrm{sgn}(u^\M)\mathbf{1}_{K_u}$ with $K_u$ such that $\max_{K\in\T}|u_K| = |u_{K_u}|$. Then 
\eqref{eq:sobdebase} yields the inequality.
\smallskip

\noindent\emph{Step 2:} It remains to show Case C for the boundary exponent $q = pd/(d-p)$ with $p<d$. In this step, we assume $p>1$. One has 
\[
\iint_{\Omega\times\Omega} \frac{|\nabla^\M u(\xi)f(x)|}{|x-\xi|^{d-1}}\dd x\dd \xi\leq C^{\rm HLS}_{p,q'}\|f\|_{L^{q'}(\Omega)} |u|_{1,p}.
\]
where $C^{\rm HLS}_{p,q'}$ is the best constant in the Hardy--Littlewood--Sobolev inequality. By \cite[Theorem 4.3]{liebloss} one has the following bound 
\[C^{\rm HLS}_{p,q'} \leq d\Bd^{1-\frac1d}\frac{1}{pq'}\left(\left(\frac{p'(d-1)}{d}\right)^{1-\frac1d} + \left(\frac{q(d-1)}{d}\right)^{1-\frac1d}\right)\leq C_d\Lambda_{p,q},
\]
for a constant $C_d$ depending only on $d$. Proceeding as in Step 1, this proves the last case.

\smallskip

\noindent\emph{Step 3:} It remains to treat the case of $p=1$ and $q = \frac{d}{d-1}$. This last case is treated as in \cite{BCCHF}, independently from the potential estimate. Indeed, observe that $u^\M\in BV(\Omega)$, and  $|u|_{1,1}$ is exactly the total variation of $u^\M$, we use the $BV(\Omega)\subset L^{\frac{d}{d-1}}(\Omega)$ embedding which yields
\begin{equation}\label{eq:BVestim}
\|u - \bar{u}\|_{0,\frac{d}{d-1}}\leq C^{\rm BV} |u|_{1,1},
\end{equation}
for some constant $C^{\rm BV}$ depending only on $\Omega$, see for example \cite[Theorem~5.11.1]{Ziemer89}. In particular since $\|u\|_{0,d/(d-1)}\leq|\Omega|^{-\frac1d}\|u\|_{0,1} + \|u - \bar{u}\|_{0,d/(d-1)}$, one obtains the desired inequality.
\end{proof}

Theorem~\ref{thm:Sobembedding} is then a consequence of Proposition~\ref{prop:preciseembedding}.

\begin{proof}[Proof of Theorem~\ref{thm:Sobembedding}]
The inequality follows directly from Proposition~\ref{prop:preciseembedding}. Indeed observing that $\Lambda_{p,q}\geq1$ and applying H\"older inequality for the first term of the right-hand side of \eqref{eq:preciseemb1} one finds \eqref{eq:Sobemb} with 
\begin{equation}\label{eq:Csob}
C^{\rm Sob}_{p,q}  = C \diam(\Omega)^{1 +\frac{d}q-\frac{d}p }\Lambda_{p,q},
\end{equation}
for some $C$ depending only on $\Omega$ and $\lambda$.
\end{proof}

\subsection{Poincaré--Sobolev inequalities: proof of Theorems~\ref{thm:PoincarSob} and \ref{thm:PoincarSobDir}}\label{sec: proofs Thm Poincare}
\begin{proof}[Proof of Theorem~\ref{thm:PoincarSob}] As in the third step of the proof of Proposition~\ref{prop:preciseembedding}, we use the $BV(\Omega)\subset L^{\frac{d}{d-1}}(\Omega)\subset L^{1}(\Omega)$ embeddings which yield thanks to \eqref{eq:BVestim} that
\[ 
\|u - \bar{u}\|_{0,1}\leq |\Omega|^{\frac1d}C^{\rm BV} |u|_{1,1}.
\]
Applying Proposition~\ref{prop:preciseembedding} to $u-\bar{u}$ and injecting the above inequality combined with H\"older's inequality yields the result with $C^{0}_{p,q} = C \diam(\Omega)^{1 +\frac{d}q-\frac{d}p } \Lambda_{p,q}$  for some $C$ depending only on $\Omega$ (through $C^{\rm BV}$ in particular) and $\lambda$.
\end{proof}
\begin{proof}[Proof of Theorem~\ref{thm:PoincarSobDir}]
Combining H\"older inequality, \cite[Lemma 4.1]{BCCHF} and Remark~\ref{rk:equivseminorms} we have
\begin{equation*}
    \|u\|_{0,1} \leq |\Omega|^{\frac1d} \|u\|_{0,\frac{d}{d-1}} \leq C^{\rm BCF} |\Omega|^{\frac1d} N_{1,1,\Gamma_0}(u) = C^{\rm BCF} |\Omega|^{\frac1d} |u|_{1,1,\Gamma_0} ,
\end{equation*}
for some constant $C^{\rm BCF}$ depending only on $\Omega$ and $\Gamma_0$. Reinserting the above inequality into the result of Proposition~\ref{prop:preciseembedding} and proceeding as in the proofs of the two previous theorems, one finds $C^{\Gamma_0}_{p,q} = C \diam(\Omega)^{1 +\frac{d}q-\frac{d}p } \Lambda_{p,q}$  for some $C$ depending only on $\Omega$ and $\Gamma_0$  (through $C^{\rm BCF}$ in particular) as well as $\lambda$.
\end{proof}
\subsection{Gagliardo--Nirenberg--Sobolev inequality: proof of Theorem~\ref{thm:GagliardoNirenbergGeneral}}\label{sec: Gagliardo--Nirenberg proof} 
In order to prove Theorem~\ref{thm:GagliardoNirenbergGeneral} for the critical and super-critical exponents $p\geq d$, we need to refine and localize the estimates of Section~\ref{sec: discrete potential inequality}. More precisely we are going to derive a discrete version of \cite[Lemma~4.15]{AdamsFournier2003}, which is \eqref{eq:loc_estimate_discrete} below. In the following, for the sake of conciseness, we sometimes use $\lesssim$ and $\gtrsim$ to denote inequalities up to some constants depending only on $\Omega$, $\lambda$ and $\mu$. Given $\rho\in(0,1]$, let us introduce 
\begin{equation}\label{eq:AKsig_loc}
A_{K\sigma}^\rho = \frac{|S_{K\sigma}\cap k_\Omega^\rho(x_K)|}{|\sigma|}, 
\end{equation}
with the notation $ k_\Omega^\rho(x_K)=\{x_K+\rho (y-x_K), \, y \in k_\Omega(x_K)\}$. Then, $A_{K\sigma}^\rho$ is clearly a non-decreasing function of $\rho$ and $A_{K\sigma}^\rho\leq A_{K\sigma}$. From there, a direct adaptation of Lemma~\ref{lem:ineqAksig} yields that for all $K\in\T$, one has 
\begin{equation}\label{eq:estim_loc}
   |u_K| \le \frac{1}{|k_\Omega|\rho^{d}}\int_{ k_\Omega^\rho(x_K)} | u^\M| + \frac{1}{|k_\Omega|\rho^{d}}\sum_{\substack{\sigma\in\Eint\\ \sigma \in \EKpLp}} A_{K\sigma}^\rho \frac{|\sigma||u_{K'}-u_{L'}|}{|D_\sigma|} |D_\sigma|.
\end{equation}
Then, we will estimate $A_{K\sigma}^\rho$ following the lines of Lemma~\ref{lem.Aksigma} and Lemma~\ref{lem.Aksigma_reg}. 

\begin{lemma}\label{lem.Aksigma_loc}
Let $\M$ be a mesh in the sense of {\normalfont \ref{HM a partition}--\ref{HM e cell centers}} satisfying the regularity assumptions {\normalfont \ref{HR a volumes}--\ref{HR b faces}}. Then, there exist explicit constants $C$, $\gamma>0$ that depend only on $\Omega$, $\lambda$ and $\mu$ such that for any $K\in\T$ and $\sigma\in \Eint$ and for any $\rho\in(0,1]$, a.e. $x\in K$ and a.e. $\xi\in D_\sigma$ we have
\begin{equation}\label{eq.Aksigma_lo_estimate}
      A_{K\sigma}^\rho \le \frac{C \rho^d}{|x-\xi|^{d-1}}\mathbf{1}_{B(0,\gamma\rho)}(x-\xi).
\end{equation}
\end{lemma}

\begin{proof}
 First, observe that if $\dd(x_K,\widehat{\sigma})>\rho\,\diam(k_\Omega)$ then $A_{K\sigma}^\rho = 0$. Moreover, by \ref{HR a volumes}, if $\rho\,\diam(k_\Omega) \leq\mu h$ then $A_{K\sigma}^\rho = 0$. Therefore, the bound is valid in these cases and we may assume from now on that $\dd(x_K,\widehat{\sigma})\leq\rho\,\diam(k_\Omega)$ and $\rho\,\diam(k_\Omega) >\mu h$. Following Lemma~\ref{lem.Aksigma_reg} we split the estimate depending on the size of $|\xi-x|$.
 
\noindent\emph{Step 1: $|\xi-x|\le 4 h$.} One has
\[
A_{K\sigma}^\rho\leq \frac{| k_\Omega^\rho|}{|\sigma|} \lesssim \frac{\rho^d}{|x-\xi|^{d-1}}.
\]
\noindent\emph{Step 2: $|\xi-x| > 4 h$.} In this case we recall that $\dd(x_K,\widehat{\sigma})>0$. A direct adaptation of the proof of Lemma~\ref{lem.Aksigma} yields that 
\begin{equation}\label{lem.Aksigma.localized}
 A_{K\sigma}^\rho \lesssim \frac{\rho^d}{|\sigma|}\, \left(\dfrac{ h_\sigma}{\dd(x_K,\widehat{\sigma})}\right)^{d-1}.
\end{equation}
In turn, since $\dd(x_K,\widehat{\sigma})>|x-\xi|/2$ in this case, one infers the estimate
\[
A_{K\sigma}^\rho \lesssim \frac{h^{d-1}\rho^d}{|\sigma|\dd(x_K,\widehat{\sigma})^{d-1}}\lesssim \frac{\rho^d}{|x-\xi|^{d-1}}.
\]
\noindent\emph{Step 3: localization.} We have found at this point that the bound $A_{K\sigma}^\rho\lesssim \rho^d|x-\xi|^{1-d}$ holds provided that $\rho\,\diam(k_\Omega) >\mu h$ and $\dd(x_K,\widehat{\sigma})\leq\rho\,\diam(k_\Omega)$, and otherwise $A_{K\sigma}^\rho=0$. Observe that if $\rho\,\diam(k_\Omega) >\mu h$ and $|x-\xi|\geq \rho \,\diam(k_\Omega)(1+2\mu^{-1})$, then $|x-\xi|>\rho\,\diam(k_\Omega)+2h$. But by the triangle inequality one always has $|x-\xi|\leq \dd(x_K,\widehat{\sigma}) + 2h$, thus in this case $\dd(x_K,\widehat{\sigma})>\rho\,\diam(k_\Omega)$ and $A_{K\sigma}^\rho=0$. This yields the claim with $\gamma = \diam(k_\Omega)(1+2\mu^{-1})$.
\end{proof}

This allows us to derive the following localized version of Proposition~\ref{prop:preciseembedding}. 
\begin{proposition}\label{prop:evenmorepreciseembedding}
 Let $\M$, $p$, $q$ and $r$ be as in Theorem~\ref{thm:GagliardoNirenbergGeneral} and additionally such that $p\geq d$ and $q<\infty$ if $p=d$. Then, for all $u\in\R^\T$, and $\rho\in(0,1]$ one has 
\[
\|u\|_{0,q}\leq C\left(\rho^{\frac dq - \frac dr}\|u\|_{0,r} + \rho^{\frac dq - \frac dp + 1}|u|_{1,p}\right),
\]
where $C$ depends only on $\Omega$, $\lambda$, $\mu$, $p$, $q$, and $r$.
\end{proposition}
\begin{proof}
Under the assumptions, observe that $1\leq r < q$ and $\frac1q -\frac1p+\frac1d>0$. Let us distinguish the cases of small and large $\rho>0$. Assume that $(\gamma-\diam(k_\Omega))\rho\leq  h$, with $\gamma>\diam(k_\Omega)$ given in Lemma~\ref{lem.Aksigma_loc}. First, since $r\leq q$, by sub-additivity of $x\to x^{r/q}$, one has 
\[
\|u\|_{0,q} \leq \left(\min_{K\in\T} |K|\right)^{\frac1q-\frac1r}\|u\|_{0,r}.
\]
But, since $r\leq q$ and  $|K|\geq \mu^d h^d|B_d|\geq\mu^d|B_d|(\gamma-\diam(k_\Omega))^d\rho^d$, the desired inequality holds. Now we assume that $\rho\,\diam(k_\Omega) + h \leq \gamma\rho$. In this case observe that for any $x\in K$, one has $k_\Omega^\rho(x_K)\subset \Omega\cap B(x,\gamma\rho)$. Therefore, using \eqref{eq:estim_loc} and Lemma~\ref{lem.Aksigma_loc}, one gets
\begin{equation}\label{eq:loc_estimate_discrete}
|u^\M(x)|\lesssim \rho^{-d}\int_{B(x,\gamma\rho)\cap\Omega} |u^\M(\xi)|\,\dd\xi  + \int_{B(x,\gamma\rho)\cap\Omega} \frac{ |\nabla^\M u(\xi)|}{|x-\xi|^{d-1}}\,\dd\xi.
\end{equation}
Then, the result follows from Young's convolution inequality used on the two terms.
\end{proof}
We are now equipped to prove the discrete Gagliardo--Nirenberg--Sobolev inequality.
\begin{proof}[Proof of Theorem~\ref{thm:GagliardoNirenbergGeneral}] We divide the proof in two steps depending on the value of $p$. We argue similarly in the cases of zero mean value and Dirichlet boundary conditions and do not detail the proofs in these cases.

\smallskip

\noindent\emph{Step 1:} Let $1\leq p <d$ and define $p^* = pd/(d-p)$. First, using H\"older inequality with exponents $p^*/\theta q$ and $r/(1-\theta)q$,  it holds that
\begin{align*}
    \|u\|^q_{0,q} \leq \|u\|_{0,p^*}^{\theta q} \, \|u\|_{0,r}^{(1-\theta)q}.
\end{align*}
Hence, applying Theorem~\ref{thm:Sobembedding} Case C, we obtain \eqref{ineg GNS} with the constant $\left(C^{\rm Sob}_{p,p^*}\right)^{\theta}$ where $C^{\rm Sob}_{p,p^*}$ is defined in \eqref{eq:Csob}. 

\smallskip

\noindent\emph{Step 2:} Let $p=d$ and $q<\infty$ (so that $\theta<1$) or $p > d$. Under these assumptions, we recall that $1\leq r < q$ and $\frac1q -\frac1p+\frac1d>0$. Let 
\[A = \|u\|_{0,r}\quad\text{and}\quad B = C^\text{Sob}_{p,r}\|u\|_{1,p} .\] 
Using Proposition~\ref{prop:evenmorepreciseembedding}, one obtains for a possibly worse constant than in the proposition that
\[
\|u\|_{0,q}\leq C\left(A\left(\tfrac dr - \tfrac dq\right)^{-1}\rho^{\frac dq - \frac dr} + B \left(\tfrac dq - \tfrac dp + 1\right)^{-1}\rho^{\frac dq - \frac dp + 1}\right) =: F(\rho).
\]
The function $F: (0,+\infty)\to (0,+\infty)$ attains its minimum at $\rho_* = (A/B)^{(1+\frac dr - \frac dp)^{-1}}$ with $\rho_* \leq 1$ thanks to Theorem~\ref{thm:Sobembedding}. By taking $\rho = \rho_*$ in the above inequality, one obtains the desired result.
\end{proof}

\subsection{Trudinger inequality: proof of Theorem~\ref{thm:Trudinger1} and extensions}\label{sec: Trudinger}

In this section, we first prove Theorem~\ref{thm:Trudinger1} as well as alternative formulations in Proposition~\ref{prop:Trudinger2}, Proposition~\ref{prop:Trudinger3}, and Proposition~\ref{prop:Trudinger4}.

\begin{proof}[Proof of Theorem~\ref{thm:Trudinger1}]
Let $\alpha>0$ and $u\in\R^\T$ such that $\|u\|_{1,d}\leq1$. Using the series expansion of the exponential function and setting $d' = \frac{d}{d-1}$, one has
\begin{equation*}
\sum_{K\in\T}|K|\left[\exp\left(\alpha |u_K|^{d'}\right)-1\right] = \sum_{k=1}^\infty \frac{\alpha^k}{k!}\|u\|_{0,kd'}^{kd'}\leq \sum_{k=1}^\infty \frac{\alpha^k}{k!} \left(C_{d,kd'}^\text{Sob}\right)^{kd'},
\end{equation*}
where we used Theorem~\ref{thm:Sobembedding}. Now using \eqref{eq:Csob}, we have, for $k\ge d-1$,
\[
    \left(C^{\rm Sob}_{d,kd'}\right)^{kd'} = C^{kd'}\, \diam(\Omega)^d (k+1)^{k+1}.
\]
Therefore, we obtain
\begin{equation*}
\sum_{K\in\T}|K|\left[\exp\left(\alpha |u_K|^{d'}\right)-1\right]\leq \diam(\Omega)^d \sum_{k=1}^\infty a_k,
\end{equation*}
where $a_k = \alpha^kC^{kd'}(k+1)^{k+1}/k!$ for $k\ge d-1$. Since $a_{k+1}/a_k\to \alpha e C^{d'}$ as $k\to \infty$ the series converges if $\alpha < \alpha_\star = e^{-1} C^{-d'}$ by d'Alembert criterion. 
\end{proof}

\begin{proposition}[Trudinger inequality, second form]\label{prop:Trudinger2}
Under the assumptions of Theorem~\ref{thm:Trudinger1}, one has with the same constants that for all $0<\beta<\beta_\star$, there is a constant $C_\beta$ depending only on $\Omega$, $\lambda$ and $\beta$ such that for any $u\in\R^\T$,
\begin{equation}\label{eq:MT2}
\sum_{K\in\T}|K|\exp|u_K| \leq C_\beta \exp\left(\frac{1}{\beta}\|u\|_{1,d}^d\right),
\end{equation}
where $\beta_\star = \frac{\alpha_\star^{d-1}d^d}{(d-1)^{d-1}}$ depends only on $\Omega$ and $\lambda$.
\end{proposition}
\begin{proof} If $u=0$, the inequality is satisfied. Otherwise, let us set $v = u/\|u\|_{1,d}$. By Young's inequality one has that for all $K\in\T$
\[
|u_K| \leq \alpha|v_K|^{d'} + \frac{\|u\|^{d}_{1,d}}{d(\alpha d')^{d-1}}.
\]
Then, we may apply \eqref{eq:MT1} to conclude with $\beta =  \frac{\alpha^{d-1}d^d}{(d-1)^{d-1}}$.
\end{proof}

\begin{proposition}[Trudinger inequality for zero average functions]\label{prop:Trudinger3}
Under the assumptions of Theorem~\ref{thm:Trudinger1}, there exists $\alpha^0_\star$ such that for all $0<\alpha<\alpha_\star^0$, there is $C_\alpha>0$ such that
\begin{equation}\label{eq:MT4}
\sup_{\substack{u\in\R^\T\\|u|_{1,d}\leq 1}} \sum_{K\in\T}|K|\exp\left(\alpha |u_K- \bar{u}|^{\frac{d}{d-1}}\right)\leq C_\alpha,\qquad \bar{u} = \frac{1}{|\Omega|}\sum_{K\in\T} |K|u_K.
\end{equation}
There also exists $\beta_\star^0$, such that for all $0<\beta<\beta_\star^0$, there is $C_\beta>0$ such that for all $u\in\R^\T$
\begin{equation}\label{eq:MT3}
\sum_{K\in\T}|K|\exp|u_K - \bar{u}| \leq C_\beta \exp\left(\frac{1}{\beta}|u|_{1,d}^d\right),\qquad \bar{u} = \frac{1}{|\Omega|}\sum_{K\in\T} |K|u_K.
\end{equation}
The constants depend on the mesh only through $\Omega$ and $\lambda$.
\end{proposition}

\begin{proof} The proof is obtained as a combination of the proofs of Theorem~\ref{thm:Trudinger1} and  Proposition~\ref{prop:Trudinger2} starting from Theorem~\ref{thm:PoincarSob}.
\end{proof}

\begin{proposition}[Trudinger inequality with Dirichlet boundary conditions]\label{prop:Trudinger4} Let $\Gamma_0\subset \partial \Omega$ be a non-empty portion of the boundary with $\Gamma_0=\cup_{\sigma \in \Eextz} \sigma$ for some $\Eextz\subset \Eext$. Then, under the assumptions of Theorem~\ref{thm:Trudinger1}, there exists $\alpha^{\Gamma_0}_\star$ such that for all $0<\alpha<\alpha_\star^{\Gamma_0}$, there is $C_\alpha>0$ such that
\begin{equation}\label{eq:MT5}
\sup_{\substack{u\in\R^\T\\ |u|_{1,d,\Gamma_0}\leq 1}} \sum_{K\in\T}|K|\exp\left(\alpha |u_K|^{\frac{d}{d-1}}\right)\leq C_\alpha.
\end{equation}
There also exists $\beta_\star^{\Gamma_0}$ such that for all $0<\beta<\beta_\star^{\Gamma_0}$, there is $C_\beta>0$ such that for all $u\in\R^\T$
\begin{equation}\label{eq:MT6}
\sum_{K\in\T}|K|\exp|u_K| \leq C_\beta \exp\left(\frac{1}{\beta} |u|_{1,d,\Gamma_0}^d \right).
\end{equation}
The constants depend on the mesh only through $\Omega$, $\Gamma_0$, and $\lambda$.
\end{proposition}

\begin{proof} The proof is obtained as a combination of the proofs of Theorem~\ref{thm:Trudinger1} and  Proposition~\ref{prop:Trudinger2} starting from Theorem~\ref{thm:PoincarSobDir}.
\end{proof}

\section{Extension to broken Sobolev spaces}\label{sec:HFV}

In the following, we focus on the generalization of our results to broken Sobolev spaces, leading to discrete functional inequalities usable in the framework of high-order numerical approximations such as discontinuous Galerkin (DG) schemes. We discuss briefly the extension to hybrid non-conforming schemes in Remark~\ref{rem:hybrid}.

Given a mesh $\M=(\T,\E,\Pcal)$ in the sense of {\normalfont \ref{HM a partition}--\ref{HM e cell centers}}, satisfying the regularity assumptions {\normalfont \ref{HR a volumes}--\ref{HR b faces}}, and with mesh size $h$, we introduce the broken Sobolev space 
\[
W^{1,p}(\T) = \{u\in L^p(\Omega)\ :\quad u_{|K} \in W^{1,p}(K),\ \forall K\in \T\}\,.
\]
For functions $u\in W^{1,p}(\T)$, we also introduce the associated broken seminorm
\[
|u|_{W^{1,p}(\T)} = \left(\sum_{K\in\T}\|\nabla u\|_{L^p(K)}^p\right)^{\frac1p},\quad \text{for } p\in[1,\infty),\quad \text{and}\quad |u|_{W^{1,\infty}(\T)} = \max_{K\in\T} \|\nabla u\|_{L^\infty(K)},
\]
as well as the DG penalization term 
\begin{align*}
J_p(u) = \left(\sum_{\sigma\in\Eint}\frac{1}{\dd_\sigma^{p-1}}\|[u]\|_{L^{p}(\sigma)}^p\right)^{\frac{1}{p}},\quad \text{for } p\in[1,\infty),\quad \text{and}\quad J_\infty(u) =  \max_{\sigma\in\Eint}\dd_\sigma^{-1}\|[u]\|_{L^{\infty}(\sigma)},
\end{align*}
where $[u](x) = |u_{|K}(x) - u_{|L}(x)|$ denotes the jump of the traces at almost every $x\in\sigma$, $\sigma\in\EKL$. From there, the DG seminorm reads
\[
|u|_{{\rm DG},p} = |u|_{W^{1,p}(\T)} + J_p(u),
\]
and the corresponding norm
\[
\|u\|_{{\rm DG},p} = |u|_{{\rm DG},p} + \|u\|_{L^p(\Omega)}.
\]
In the following, we use the notation $u_E$ to denote the average of $u\in W^{1,p}(\T)$ over a measurable subset $E\subset\Omega$. For later use, we define $\Pi:W^{1,p}(\T)\to \R^\T$ the projection operator defined by $\Pi(u)_K = u_K$.

We will show that the following high-order extension of Theorem~\ref{thm:Sobembedding} and Theorem~\ref{thm:Trudinger1} holds in broken Sobolev spaces.
\begin{theorem}[Discrete Sobolev and Trudinger inequalities in broken Sobolev spaces]\label{thm:DG}
    Under the assumptions of Theorem~\ref{thm:Sobembedding}, there holds for all $u\in W^{1,p}(\T)$
     \begin{equation}\label{eq:SobembDG}
 \|u\|_{L^q(\Omega)}\leq C^{\rm DG}_{p,q}\|u\|_{{\rm DG},p},
 \end{equation}
  with a constant $C^{\rm DG}_{p,q}$ depending only on $\Omega$, $\mu$, $\lambda$, $p$ and $q$ which satisfies the asymptotics~\eqref{eq:asymp}. Moreover, there exists an explicit constant $\alpha_\star^{\rm DG}>0$ depending only on  $\Omega$, $\lambda$ and $\mu$ such that for all $0\leq\alpha<\alpha_\star^{\rm DG}$, there is $C_\alpha>0$ which depends only on  $\Omega$, $\lambda$, $\mu$ and $\alpha$ such that
 \begin{equation}\label{eq:MT_DG}
\sup_{\substack{u\in W^{1,d}(\T)\\\|u\|_{{\rm DG},d}\leq 1}} \int_\Omega\exp\left(\alpha |u|^{\frac{d}{d-1}}\right)\leq C_\alpha.
\end{equation}
\end{theorem}

Let us emphasize that the above result holds for general piecewise Sobolev functions, i.e., we do not need to assume that the functions are polynomials on each cell of $\T$. The rest of the section is dedicated to the proof of Theorem~\ref{thm:DG}. It is essentially a corollary of Theorem~\ref{thm:Sobembedding} and Theorem~\ref{thm:Trudinger1}, combined with results that are recalled in the following preliminary lemmas. Overall, most elements of the proof of \eqref{eq:SobembDG} can be found in the literature and we refer the reader to the books \cite{fem_book, dg_book, DroniouEtAl2018, hho_book} and to the papers \cite{arnold02, brenner03, buffa09} for more details. In contrast to the literature, here, we carefully track the dependency of constants in Lebesgue exponents in order to recover the asymptotics~\eqref{eq:asymp}. Moreover, this allows for the derivation of a new Trudinger inequality \eqref{eq:MT_DG} in the DG framework.

\begin{remark}[Discussion on the definition of the DG penalization term]
By defining alternatively the DG penalization term as
\[
J^*_p(u) = \left(\sum_{\sigma\in\Eint}\frac{1}{|D_\sigma|^{p-1}} \left| \int_\sigma [u] \right|^p\right)^{\frac{1}{p}},\quad \text{for } p\in[1,\infty),\quad \text{and}\quad J^*_\infty(u) =  \max_{\sigma\in\Eint}\left|D_\sigma\right|^{-1} \left| \int_\sigma [u] \right|,
\] 
we notice that for any $u \in W^{1,p}(\T)$, $|\Pi(u)|_{1,p} = J^*_p(\Pi(u))$. In particular, one can notice that the end of the proof of~\eqref{eq:SobembDG} (see below) consists in showing that $J_p^*(\Pi(u)) \leq C |u|_{{\rm DG},p}$ for some constant depending only on $d$, $\mu$ and $\lambda$. Therefore, defining an alternative DG norm associated to $J_p^*$ instead of $J_p$, and following the proof of Theorem~\ref{thm:DG}, yields a more refined estimate than~\eqref{eq:SobembDG}. However, since in 
the literature most authors work with the definition given by $J_p$ (as it is the quantity arising in the analysis of DG schemes), we state and prove the Sobolev inequalities with this definition of the DG penalization term.

Finally, thanks to Lemma~\ref{lem:regularity}, we also notice that we could alternatively define the DG penalization term $J_p$ with $h_\sigma$ instead of $\dd_\sigma$.
\end{remark}

The following lemma is a direct adaptation of \cite[Lemma 1.31]{hho_book} and earlier references therein. We only sketch its proof.

\begin{lemma}[Trace inequality in a pyramid]\label{lem:trace}
Under the assumptions of Theorem~\ref{thm:Sobembedding}, there is a constant $C>0$ depending only on $d$, $\mu$ and $\lambda$ such that for any $p\in[1,\infty]$, $K\in\T$, $\sigma\in\E_K$ and $u\in W^{1,p}(K)$, one has 
\begin{equation}
\|u\|_{L^p(\sigma)}\leq C h^{-\frac{1}{p}}\left(\|u\|_{L^p(P_{K\sigma})} + h\|\nabla u\|_{L^p(P_{K\sigma})}\right).
\end{equation}
\end{lemma}

\begin{proof} Let $p<+\infty$, the $L^\infty$ case trivially holds since in this case $u$ is continuous and $\|u\|_{L^\infty(\sigma)}\leq \|u\|_{L^\infty(P_{K\sigma})}$. Let $\phi_{K\sigma}:P_{K\sigma}\to\R^d$ be defined by $\phi_{K\sigma}(x) = \frac{|\sigma|}{d|P_{K\sigma}|}(x-x_K)$. Using the fact that on $\partial P_{K\sigma}\setminus \sigma$, $\phi_{K\sigma}$ is almost everywhere orthogonal to the normal vector of $\partial P_{K\sigma}$, and that $(x-x_K)\cdot n_{K\sigma} = \dd(x_K,H_\sigma)$ for $x\in\sigma$, one has by the divergence theorem that
\[
\|u\|_{L^p(\sigma)}^p = \int_{P_{K\sigma}}|u|^p\nabla\cdot\phi_{K\sigma} + p\int_{P_{K\sigma}}{\rm sign}(u)|u|^{p-1}\phi_{K\sigma}\cdot\nabla u\,,
\]
which leads to, using H\"older inequality and the definition of $\phi_{K\sigma}$,
\[
\|u\|_{L^p(\sigma)}^p\leq\frac{|\sigma|}{|P_{K\sigma}|}\|u\|_{L^p(P_{K\sigma})}^p + p \frac{h|\sigma|}{d|P_{K\sigma}|}\|\nabla u\|_{L^p(P_{K\sigma})}\|u\|_{L^p(P_{K\sigma})}^{p-1}.
\]
Using first the properties of Lemma~\ref{lem:regularity} to bound geometric quantities, and then Young's inequality $ab\leq \tfrac{a^p}{p} + \tfrac{b^{p'}}{p'}$ for the second term of the right-hand side with $a=h\|\nabla u\|_{L^p(P_{K\sigma})}$ and $b = \|u\|_{L^p(P_{K\sigma})}^{p-1}$, one finds, for some constant $C$ depending only on $d,\mu,\lambda$ that 
\[
\|u\|_{L^p(\sigma)}^p\leq C h^{-1}\left((d+p-1)\|u\|_{L^p(P_{K\sigma})}^p + h^p\|\nabla u\|_{L^p(P_{K\sigma})}^p\right).
\]
This leads to the desired inequality for a different constant $C$, still depending only on $d$, $\mu$, and $\lambda$.
\end{proof}

\begin{lemma}[Poincaré--Sobolev-Wirtinger in a cell]\label{lem:Poincare_cell}
Under the assumptions of Theorem~\ref{thm:Sobembedding} and assuming moreover that $p>1$ when $\frac1d+\frac{1}{q}-\frac{1}{p} = 0$, there is a constant $C_d$ which depends only on $d$, such that for any $K\in\T$ and any $u\in W^{1,p}(K)$, one has 
\[
\|u - u_K\|_{L^q(K)}\leq C_d\frac{|K|^{\frac1d+\frac{1}{q}-\frac{1}{p}}}{\mu^d} \Lambda_{p,q}\|\nabla u\|_{L^p(K)}\,,
\]
where $\Lambda_{p,q}$ is defined in \eqref{eq:constantSob}.
\end{lemma}
\begin{proof}
Let us write $B=B(x_K,\mu h)$. First observe that
\[
\|u - u_K\|_{L^q(K)}\leq \|u - u_B\|_{L^q(K)} + |K|^{1/q}|u_K-u_B|,
\]
and so by H\"older inequality one has 
\[
\|u - u_K\|_{L^q(K)}\leq 2\|u - u_B\|_{L^q(K)}.
\]
Then, since $K$ is star-shaped with respect to $B$, for any $x\in K$ one has
\[
u(x) - u_B = \frac{1}{|B|}\int_B \int_0^1 \nabla u(t x + (1-t)y)\cdot(x-y)\dd t\dd y.
\]
Using the change of variable $z = tx + (1-t)y$ one finds
\[
|u(x) - u_B|\leq\frac{1}{|B|}\int_0^1\int_{tx+(1-t)B}|\nabla u(z)||x-z|\frac{1}{(1-t)^{d+1}} \dd z\dd t.
\]
Now, since $1-t = \frac{|x-z|}{|x-y|}\geq \frac{|x-z|}{h}$ one obtains,
\[
|u(x) - u_B|\leq\frac{1}{|B|}\int_K\int_0^{1-\frac{|x-z|}{h}}|\nabla u(z)||x-z|\frac{1}{(1-t)^{d+1}} \dd t\dd z,
\]
which yields after computing the integral in $t$ that
\[
|u(x) - u_B| \leq \frac{h^d}{d|B|} \int_K \frac{|\nabla u(z)|}{|x-z|^{d-1}}\dd z =\frac{1}{d \mu^d \Bd} \int_K \frac{|\nabla u(z)|}{|x-z|^{d-1}}\dd z.
\]
From there one concludes with \cite[Lemma 7.12]{GilbargTrudinger2001} in the case $\frac{1}{d}+\frac{1}{q}>\frac{1}{p}$ and with \cite[Theorem 4.3]{liebloss} in the critical case $\frac{1}{d}+\frac{1}{q}=\frac{1}{p}$.
\end{proof}

We can now turn to the proof of the main result of this section.

\begin{proof}[Proof of Theorem~\ref{thm:DG}]
Once  \eqref{eq:SobembDG} is satisfied, the proof of \eqref{eq:MT_DG} follows the lines of the proof of Theorem~\ref{thm:Trudinger1}, so that we only detail \eqref{eq:SobembDG}. We may assume that  $p>1$ if $\frac1d+\frac{1}{q}-\frac{1}{p} = 0$. Indeed, in the remaining case $p=1$ and $q = \frac{d}{d-1}$ then we have that $\|u\|_{BV(\Omega)} \leq C_d \|u\|_{{\rm DG},1}$ (see \cite[Lemma 5.2]{dg_book}) and the result follows from the $BV(\Omega)$ embedding into $L^{\frac{d}{d-1}}$. One has by Theorem~\ref{thm:Sobembedding}
\[
\|u\|_{L^q(\Omega)}\leq \|u-\Pi(u)^\M\|_{L^q(\Omega)} + \|\Pi(u)\|_{0,q}
\leq \left(\sum_{K\in\T}\|u-u_K\|_{L^q(K)}^q\right)^{\frac{1}{q}} + C \Lambda_{p,q}\|\Pi(u)\|_{1,p}.
\]
Using Lemma~\ref{lem:Poincare_cell}, the first term is estimated by
\[
\left(\sum_{K\in\T}\|u-u_K\|_{L^q(K)}^q\right)^{\frac{1}{q}}\leq C_d\frac{h^{1+\frac{d}{q}-\frac{d}{p}}}{\mu^d} \Lambda_{p,q}\left(\sum_{K\in\T}\|\nabla u\|_{L^p(K)}^q\right)^{\frac{1}{q}}.
\]
Observe that, since $\frac1d \ge \frac1d+\frac{1}{q}-\frac{1}{p}\geq 0$, the numerator above is bounded by $\max(1,\diam(\Omega))$. Moreover, since $p/q\leq 1$, we apply the sub-additivity of the function $x \mapsto x^{p/q}$ on the right-hand side terms of the last inequality to get
\[
\left(\sum_{K\in\T}\|u-u_K\|_{L^q(K)}^q\right)^{\frac{1}{q}}\leq C\Lambda_{p,q}|u|_{W^{1,p}(\T)},
\]
for a constant $C$ depending only on $\mu$ and $\Omega$. Let us now inspect the second term. For this, we introduce the face average
\[
\left\langle u_{|K}\right\rangle_\sigma = \frac{1}{|\sigma|}\int_\sigma u_{|K}\,,
\]
which we recall is well defined since $u_{|K}$ has trace in $L^p$ by Lemma~\ref{lem:trace}. In the following we only write the case $p<\infty$ and the case $p=\infty$ is treated the same way. Using Lemma~\ref{lem:regularity} and $(a+b+c)^p\leq 3^{p-1}(a^p+b^p+c^p)$, one has, for some constant depending only on $\Omega$, $\mu$ and $\lambda$ that may change from line to line that
\begin{align*}
|\Pi(u)|_{1,p}^p &= \sum_{\substack{\sigma\in\Eint\\\sigma \in\EKL}}|\sigma|^p|D_\sigma|^{1-p} |u_K-u_L|^p\\
&\leq 3^{p-1}\sum_{\substack{\sigma\in\Eint\\\sigma \in\EKL}}|\sigma|^p|D_\sigma|^{1-p}\left(\left|u_K-\left\langle u_{|K}\right\rangle_\sigma\right|^p + \left|u_L-\left\langle u_{|L}\right\rangle_\sigma\right|^p\right) + C^p 3^{p-1}J_p(u)^p.
\end{align*}
For all $K$ and $\sigma\in\E_K\cap\Eint$ one has by Jensen's inequality that
\[
\left|u_K-\left\langle u_{|K}\right\rangle_\sigma\right|^p= |\left\langle u_{|K}-u_K\right\rangle_\sigma|^p\leq\frac{1}{|\sigma|}\|u_{|K}-u_K\|_{L^p(\sigma)}^p.
\]
Therefore, by  Lemma~\ref{lem:trace}, property~\ref{CR volsig}, and the elementary inequality $(x+y)^p \leq 2^{p-1}(x^p+y^p)$ for all $x, y \geq 0$, one has 
\[
\left|u_K-\left\langle u_{|K}\right\rangle_\sigma\right|^p\leq C^p h^{-d}\left(\|u-u_K\|^p_{L^p(P_{K\sigma})} +  h^p\|\nabla u\|^p_{L^p(P_{K\sigma})} \right).
\]
Now, since $|\sigma|^p|D_\sigma|^{1-p}\leq C^p h^{d-p}$, it holds
\begin{align*}
    \sum_{\substack{\sigma\in\Eint\\\sigma \in\EKL}}|\sigma|^p|D_\sigma|^{1-p}\,\left|u_K-\left\langle u_{|K}\right\rangle_\sigma\right|^p &\leq C^p h^{d-p} \sum_{\substack{\sigma\in\Eint\\\sigma \in\EKL}} h^{-d} \left(\|u-u_K\|^p_{L^p(P_{K\sigma})} +  h^p\|\nabla u\|^p_{L^p(P_{K\sigma})} \right)\\ 
    &\leq C^p h^{-p} \sum_{K \in \T} \left(\|u-u_K\|^p_{L^p(K)} +  h^p\|\nabla u\|^p_{L^p(K)} \right),
\end{align*}
where we have used assumption~\ref{HR a volumes}. Finally, using Lemma~\ref{lem:Poincare_cell} and the fact that $\Lambda_{p,p}$ does not depend on $p$ one finds 
\[
\sum_{\substack{\sigma\in\Eint\\\sigma \in\EKL}}|\sigma|^p|D_\sigma|^{1-p}\, \left|u_K-\left\langle u_{|K}\right\rangle_\sigma\right|^p\leq C^p |u|_{W^{1,p}(\T)}^p,
\]
so that
\[
|\Pi(u)|_{1,p}\leq C |u|_{{\rm DG},p},
\]
for some constant depending only on $d$, $\mu$ and $\lambda$. It is easily seen from Jensen's inequality that, additionally,  $\|\Pi(u)\|_{0,p}$ satisfies the same bound. This concludes the proof.
\end{proof}

\begin{remark}\label{rk:DG}
All the other main results of the paper, namely Theorem~\ref{thm:PoincarSob}, Theorem~\ref{thm:PoincarSobDir}, Theorem~\ref{thm:GagliardoNirenbergGeneral}, Proposition~\ref{prop:Trudinger2}, Proposition~\ref{prop:Trudinger3}, and Proposition~\ref{prop:Trudinger4} generalize to DG norms in the framework of broken Sobolev spaces. The proofs can be adapted in the same way as for Theorem~\ref{thm:DG}. For the sake of conciseness we leave them to the reader. 
\end{remark}

\begin{remark}[Extensions to hybrid and related methods]\label{rem:hybrid}
For non-conforming numerical methods such as Discrete Duality Finite Volume (DDFV), Hybrid Finite Volume (HFV), Gradient Discretization Method (GDM), or Hybrid High-Order (HHO), a function, and more precisely its gradient, is approximated through additional degrees of freedom located generally on the cell faces or in the dual mesh. Discrete functional inequalities are essential to the analysis of these methods and are derived partly as a consequence of the piecewise constant cell-based discrete functional inequalities, as well as additional tools such as trace related inequalities \cite{ badia2026,hho_book}. The generalization in these various settings of the inequalities derived here is an interesting question that is left as an open problem. We refer for instance to \cite{BCCHF} for DDFV, to \cite{CHHLM, DroniouEtAl2018} for HFV / GDM and to \cite{hho_book} for HHO for existing examples of such generalizations.
\end{remark}

\section{Discussion and applications}\label{sec:discussions}

In the following, we first present an application of our newly derived Trudinger inequality to the analysis of a discretized Keller--Segel system in Section~\ref{sec:kellersegel}. Then, we investigate numerically the optimality of the constant $\alpha$ in our discrete Trudinger inequality \eqref{eq:MT5} in Section~\ref{sec:moser}. Finally, in Section~\ref{sec:defect}, we derive a complementary result (Proposition~\ref{prop:Sobnonembed}), quantifying the defect in the non-embedding $W^{1,d}(\Omega)\not\subset L^\infty(\Omega)$ in terms of the mesh size.

\subsection{Lower boundedness of the discrete Keller--Segel entropy}\label{sec:kellersegel}

One of the initial motivations for this paper is the numerical analysis of finite volume approximation of Keller--Segel type systems in spatial dimension $d=2$. The minimal parabolic-parabolic Keller--Segel system writes
\begin{equation}\label{KS}
\begin{aligned}
\partial_t u &=  \nabla\cdot\left(\nabla u-u\nabla v\right)\,,& x\in\Omega\,, \\
\partial_t v &= \Delta v + u\,,& x\in\Omega\,,
\end{aligned}
\end{equation}
with homogeneous Neumann boundary conditions for $u$ and $v$ on $\partial\Omega$. The analysis of this system relies on a crucial entropy estimate, yielding that the functional 
\begin{equation}\label{entropy_KS}
\mathcal{H}(u,v) = \int_\Omega u\ln u\,\mathrm{d}x -\int_\Omega uv\,\mathrm{d}x + \frac12\int_\Omega |\nabla v|^2\,\mathrm{d}x,
\end{equation}
decays along solutions. 
Because of the non-positive middle term, this decay does not automatically provide a uniform in time control of the solution on (say) $\|\nabla v\|_{L^2(\Omega)}$.
However, for an initial mass
\begin{equation*}
M = \int_\Omega u\, \mathrm{d}x,
\end{equation*}
below some positive mass threshold, as a consequence of the Moser--Trudinger inequality the entropy remains bounded from below (on arbitrary finite time intervals, and uniformly in time if a degradation term is added in the second equation), leading to global in time existence and results on the long-time behaviour \cite{NagaiSenbaYoshida97,GajZac98}. Here, we outline how to adapt this theory to a discrete setting.

There have been many works dealing with the numerical approximation of \eqref{KS} and variants, and we refer to \cite{Filbet2006} for the first analysed finite volume approximation and to the references in \cite{HTZ25} for a more recent bibliographical overview. For a finite volume discretization in dimension $d=2$, the discrete equivalent of the entropy reads, with the notations of the previous sections,
\begin{equation}\label{eq:entropKSdis}
\mathcal{H}^\M(u,v)  = \sum_{K\in\T}|K|\left(u_K\ln u_K - u_K v_K\right) +  \frac{1}{2}\sum_{\substack{\sigma\in\Eint\\\sigma \in\EKL}}\frac{|\sigma|}{\dd_\sigma} (v_L-v_K)^2.
\end{equation}
In the sequel, we assume, as for instance in~\cite{Filbet2006,HTZ25}, that the mesh $\M$, in the sense of {\normalfont \ref{HM a partition}--\ref{HM e cell centers}} and satisfying the regularity assumptions {\normalfont \ref{HR a volumes}--\ref{HR b faces}}, is in addition admissible in the sense of~\cite[Definition 9.1]{EGH00}, i.e., $[x_K,x_L]$ is orthogonal to $\sigma$ for all $\sigma\in\EKL$. Then, we notice, see Lemma~\ref{lem:equivseminorm} in Appendix~\ref{sec:appendix}, that the last term in the above discrete entropy functional can be rewritten as
\begin{equation}\label{eq:entropKSdis_second_form}
\mathcal{H}^\M(u,v)  = \sum_{K\in\T}|K|\left(u_K\ln u_K - u_K v_K\right) +  \frac{1}{4} |v|_{1,2}^2.
\end{equation}
Now, an adaptation of the argument in \cite[Lemma 3.4]{NagaiSenbaYoshida97} yields the following result, which relies on our newly derived discrete Trudinger inequality in Proposition~\ref{prop:Trudinger3}.

\begin{proposition}[Control of gradient norm by the discrete Keller--Segel entropy]\label{prop:entropyKS}
In dimension $d=2$, let $\M$ be a mesh in the sense of {\normalfont \ref{HM a partition}--\ref{HM e cell centers}} satisfying {\normalfont \ref{HR a volumes}--\ref{HR b faces}} and admissible in the sense of~\cite[Definition 9.1]{EGH00}. Let us define $M_\star = \beta_\star^0/4$, with $\beta_\star^0$ defined in Proposition~\ref{prop:Trudinger3}.
Then, for any $u,v \in\mathbb{R}^\T$, such that 
$u$ is non-negative component-wise and
\[
M = \sum_{K\in\T} |K|u_K < M_\star,
\]
one has
\[
\mathcal{H}^\M(u,v) \geq \delta |v|_{1,2}^2 -C-M \bar{v},
\]
for some $\delta>0$ and $C>0$ depending only on $M$, $\beta_\star^0$, $\Omega$ and $\lambda$.
\end{proposition}

\begin{proof}
Rewriting the first term of~ \eqref{eq:entropKSdis_second_form}, one has 
\[
\mathcal{H}^\M(u,v)= -M\sum_{K\in\T}|K|\frac{u_K}{M}\ln \left(\frac{e^{v_K-\bar{v}}}{u_K}\right) - M \bar{v} + \frac14 |v|^2_{1,2}.
\]
By Jensen's inequality and the discrete Trudinger inequality~\eqref{eq:MT3} for $\beta<\beta_\star^0$
\begin{align*}
\mathcal{H}^\M(u,v)&\geq -M\ln\left(\frac{1}{M}\sum_{K\in\T}|K|e^{v_K-\bar{v}}\right) + \frac14|v|_{1,2}^2 - M \bar{v}\\
&\geq -M\ln\left(\frac{C_\beta}{M}\exp\left(\frac{|v|_{1,2}^2}{\beta}\right)\right) + \frac14|v|_{1,2}^2 - M \bar{v}\\
&= -M\ln\left(\frac{C_\beta}{M}\right) + \left(\frac14-\frac{M}{\beta}\right)|v|_{1,2}^2 - M \bar{v}.
\end{align*}
Pick $\beta \in \left(4M, \beta_\star^0\right)$ to conclude.
\end{proof}

Several discretizations for System~\eqref{KS} with the property that the entropy functional~\eqref{eq:entropKSdis} (slightly adapted or generalized if needed) decays at each time step have been proposed, e.g. in \cite{ShenXu20,Zhou21,GutRod21,AcoGuiRod23,LuChenLiLiu24,HuaGouShe26}; see also \cite{HTZ25} for a variant, and \cite{Filbet2006,ZhouSaito17,AlBuPePo19,BaiCarHu20,ChenLiuShen22} for the parabolic-elliptic case. Proposition~\ref{prop:entropyKS} opens the way to the theoretical analysis of the long-time behavior of those numerical schemes.

Observe that the value $M_\star$ that we obtain depends on the regularity of the mesh (but not on its size). The theoretical value of the mass threshold obtained in \cite{NagaiSenbaYoshida97,GajZac98}, which, radial case taken aside, is conjectured to be optimal (see \cite{GajZac98,FuestLankeit23} and references therein), is related to the best constant $\beta$ in Trudinger inequality. This best constant, first investigated by Moser \cite{Moser1971} for Dirichlet boundary conditions (corresponding to Proposition~\ref{prop:Trudinger4} with $\Gamma_0=\partial \Omega$), was derived by Chang and Yang \cite{ChangYang88} for zero-mean functions (corresponding to Proposition~\ref{prop:Trudinger3}). We do not hope to approach this (continuous) theoretical value with the techniques presented in this paper. In particular, our argument for proving~\eqref{eq:MT3} relies on the original proof of Trudinger~\cite{Trudinger67} rather than on Moser's or Chang and Yang's arguments leading to the optimal constants~\cite{Moser1971,ChangYang88}. We do not know whether these techniques can be adapted to the discrete setting. In Section~\ref{sec:moser} we give a discussion and a numerical illustration concerning the expected optimal constant for Dirichlet boundary condition.

\subsection{Best constant in the Moser--Trudinger inequality}\label{sec:moser}

In this section, let $\Gamma_0=\partial\Omega$. The method used in this paper to derive the constant $\alpha_\star^{\Gamma_0}$ in Proposition~\ref{prop:Trudinger4} is non-optimal as it follows the original strategy of Trudinger \cite{Trudinger67}, which was later improved by Moser \cite{Moser1971}. We leave it for future work to determine whether the latter can be adapted to the discrete setting. Nonetheless, let us explore here experimentally the best constant in \eqref{eq:MT5}. For this we restrict ourselves to dimension $d=2$ and we assume, without loss of generality, that the domain $\Omega$ contains the unit ball. From there we introduce the Moser functions, defined for $k\geq2$ by 
\begin{equation}\label{eq:moserfunc}
m_k(x) = \frac{1}{\sqrt{2\pi}}\left\{
\begin{aligned}
&\sqrt{\ln(k)}&&\text{if } |x| \leq \frac{1}{k},\\
&\frac{\ln(1/|x|)}{\sqrt{\ln(k)}}&&\text{if } \frac{1}{k}<|x| < 1,\\
&0&&\text{otherwise}.
\end{aligned}\right.
\end{equation}
It is easily seen that $m_k\in H^1_0(\Omega)$, and $\|\nabla m_k\|_{L^2(\Omega)} = 1$. Moreover, one checks that 
\begin{equation}\label{eq:Ikalph}
I_k(\alpha) \coloneqq \int_\Omega \exp\left(\alpha\, m_k(x)^2\right)\dd x  = \pi k^{\frac{\alpha}{2\pi}-2} + \int_{\{k^{-1}\leq|x|<1\}} \exp\left(\alpha\, m_k(x)^2\right)\dd x + |\Omega|-\pi,
\end{equation}
and for $\alpha>0$ with $\alpha\neq 4 \pi$ and for $k\geq2$,
\[
\int_{\{k^{-1}\leq|x|<1\}} \exp\left(\alpha\, m_k(x)^2\right)\dd x \leq \frac{4\pi^2}{4\pi-\alpha}\left(1-k^{\frac{\alpha}{2\pi}-2}\right)\,.
\]
In turn, as $k\to\infty$, $I_k(\alpha) \asymp  k^{\frac{\alpha}{2\pi}-2}$ if $\alpha>4\pi$ and $I_k(\alpha)= O(1)$ if $\alpha<4\pi$. In particular, \eqref{eq:trudinger} does not hold for $\alpha>4\pi$. It is shown in \cite{Moser1971} that $4\pi$ is the best constant, that is \eqref{eq:trudinger} is satisfied for all $\alpha\in[0,4\pi]$. 

We now investigate the $4\pi$ threshold numerically. Given a mesh $\M$ of the polyhedral domain $\Omega$, we define the discrete counterpart of \eqref{eq:Ikalph}, that is 
\begin{equation}\label{eq:IkalphDiscr}
I_k^\M(\alpha) =  \sum_{K\in\T} |K| \exp\left(\alpha\, m_k(x_K)^2\right).
\end{equation}
For different meshes with decreasing mesh size, and different values of $k$ and $\alpha$ we compute $I_k^\M(\alpha)$ and report the results in Figure~\ref{fig:Moser}. More precisely, $\Omega$ is the square $[-1,1]^2$ that is discretized with $4$ regular triangular meshes with respectively $3\,706$, $23\,248$, $257\,012$ and $4\,107\,882$ cells. The values of $k$ are chosen as $10$, $30$, $100$, $300$ and $1\,000$. Concerning $\alpha$ we take the sub-critical value $2\pi$, the critical value $4\pi$ and the super-critical values $4.5\pi$ and $5\pi$. For each value of $\alpha$, $I_k^\M(\alpha)$ is computed with respect to $k$ for the different meshes. As $k\to \infty$ there may be no point $x_K$ in the ball $\{|x|\leq 1/k\}$ for the coarser meshes. In this case, we consider that the integral is not ``resolved'', indicated by hollow markers on the plot. We observe for $I_k^\M(\alpha)$ the expected behavior of \eqref{eq:Ikalph} at the continuous level: it diverges with the expected rate $O(k^{\frac{\alpha}{2\pi}-2})$ for super-critical values of $\alpha$ as long as the integral is resolved, while it appears uniformly bounded for the (sub-)critical values of $\alpha$. 
\begin{figure}
\includegraphics[width=\textwidth]{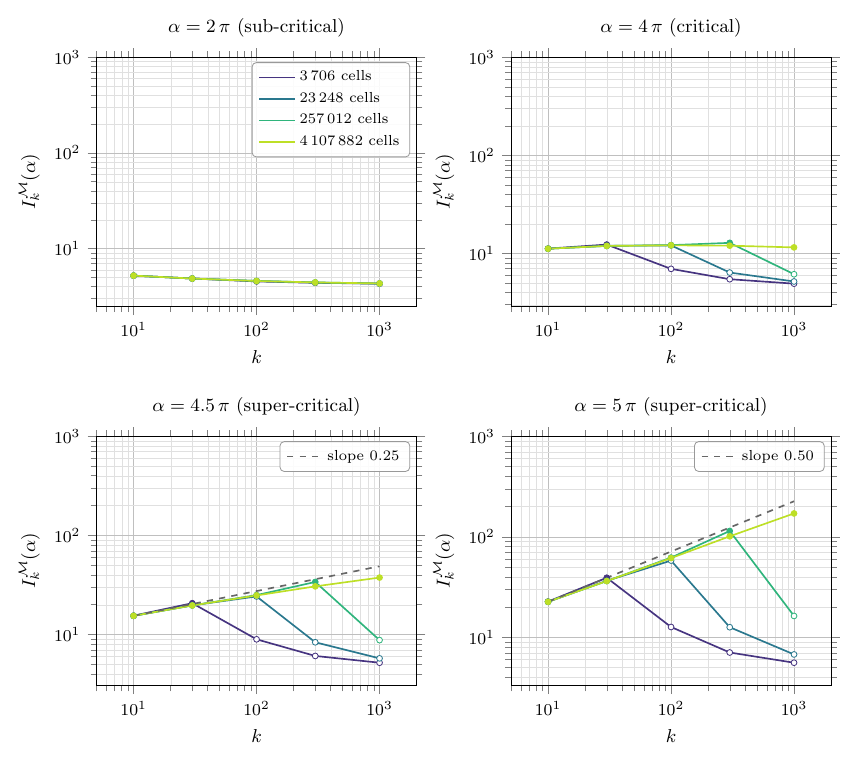}
\caption{Evolution of $I_k^\M(\alpha)$ as a function of $k$ for different values of $\alpha$ and different mesh sizes. Color-filled points indicate that the integral is resolved, that is, the ball $B(0,1/k)$ contains at least one cell center $x_K$.}
\label{fig:Moser}\end{figure}
It could be conjectured, at least for sufficiently regular meshes, that the discrete Trudinger inequality \eqref{eq:MT5} with $\Gamma_0=\partial\Omega$ holds under the same range of $\alpha$ as in the continuous setting, that is $\alpha\leq d|\mathbb{S}^{d-1}|^{1/(d-1)}$ where $\mathbb{S}^{d-1}$ is the unit sphere in dimension $d$. Let us mention that under different boundary conditions one expects different optimal constants in general (see \cite{ChangYang88}).

\subsection{Defect in discrete Sobolev non-embedding}\label{sec:defect}

The refined asymptotics of Sobolev constants given by Theorem~\ref{thm:Sobembedding} allow us to precisely quantify the blow-ups of the discrete embedding of $W^{1,d}(\Omega)$ into $L^\infty(\Omega)$ when the size of the mesh vanishes. 
This is illustrated in the following proposition.
\begin{proposition}\label{prop:Sobnonembed}
Let $\M$ be a mesh in the sense of {\normalfont \ref{HM a partition}--\ref{HM e cell centers}} satisfying the regularity assumptions {\normalfont \ref{HR a volumes}--\ref{HR b faces}}, and with mesh size $h< C_d \, \mu^{-1}$, where $C_d$ only depends on $d$. Then, there holds
\begin{equation}\label{eq:Sobdefect}
\|u\|_{0,\infty}\leq C |\ln(\mu h)|^{1-\frac1d}\|u\|_{1,d},
\end{equation}
for all $u\in\R^\T$ and some constant $C>0$ depending only on $\Omega$, $\lambda$ and $\mu$.
\end{proposition}
\begin{proof}
Let $K_0\in \T$ such that $\|u\|_{0,\infty} = |u_{K_0}|$ and $q\in[d,\infty)$. Then, with $C>0$ depending only on $\Omega$, $\lambda$ and $\mu$ and changing from an inequality to the next,
\[
\|u\|_{0,\infty}\leq |K_0|^{-\frac1q}\|u\|_{0,q}\leq C\frac{C^{\rm Sob}_{d,q}}{(\mu h)^{\frac{d}{q}}}\|u\|_{1,d}\leq C\frac{q^{1-\frac1d}}{(\mu h)^{\frac{d}{q}}}\|u\|_{1,d},
\]
where we used that $B(x_{K_0},\mu h)\subset K_0$ and Theorem~\ref{thm:Sobembedding}. By optimizing in $q$ one finds the result.
\end{proof}
We point out that the rate $ |\ln(\mu h)|^{1-\frac1d}$  improves $ |\ln(\mu h)|$ that would be found using previously known results such as \cite{BCCHF}. Such an improved rate is crucial in some applications. We refer for instance to the convergence analysis of the finite volume discretization of the Gross--Pitaevskii equation \cite{chauleur25}.

\begin{appendix}

\section{Auxiliary results}\label{sec:appendix}

In this section, we collect some auxiliary results that help put in perspective our main assumptions, objects, and theorems. We start by stating consequences on a mesh $\M$ satisfying the regularity assumptions {\normalfont \ref{HR a volumes}--\ref{HR b faces}}.

\begin{lemma}\label{lem:regularity} Let $\M$ be a mesh in the sense of {\normalfont \ref{HM a partition}--\ref{HM e cell centers}} satisfying the regularity assumptions {\normalfont \ref{HR a volumes}--\ref{HR b faces}}. Then, $\M$ satisfies the following properties:\smallskip
\begin{enumerate}[label={\normalfont (CR.{\alph*})}]
    \item \label{CR distK} For all $K\in\T$ and $\sigma\in\E_K$, we have $\mu h \le \dd(x_K,{H_\sigma})\le \dd(x_K,\sigma) \le h$.\smallskip
    \item \label{CR dsig} For all $\sigma\in\Eint$, $\sigma \in\EKL$ we have $2\mu h \le \dd(x_K,H_\sigma) +  \dd(x_L,H_\sigma) \le \dd_\sigma \le 2h$,  and for all $\sigma\in\Eext$, we have $\mu h \le \dd_\sigma \le h$.\smallskip  
    \item \label{CR diamcell} For all $K\in\T$, we have $2 \mu h \le h_K \le h$.\smallskip
    \item \label{CR volcell} For all $K\in\T$, we have $\mu^d h^d \Bd \le |K| \le \Bd \left(\frac{h}2\right)^d$.\smallskip  
    \item \label{CR diamsig} For all $\sigma\in\E$, we have $ 2 \lambda h \le h_\sigma \le  h$.\smallskip
    \item \label{CR volsig} For all $\sigma\in\E$, we have $\lambda^{d-1} h^{d-1} \Bdm \le |\sigma| \le \Bdm \left(\frac{h}{2}\right)^{d-1}$.\smallskip
\end{enumerate}
Furthermore, for any $K\in\T$ the pyramids $P_{K\sigma}$ with $\sigma\in\E_K$ partition $K$ in the sense \eqref{eq.partition}, and they satisfy
\smallskip
\begin{enumerate}[label={\normalfont (CR.{\alph*})}]
\setcounter{enumi}{6}
    \item \label{CR diampyr} For all $K\in\T$ and $\sigma\in\E_K$, we have $\max(\mu,2\lambda) h \le \diam(P_{K\sigma}) \le h$.
    \item \label{CR volpyr} For all $K\in\T$ and $\sigma\in\E_K$, we have $\frac{\mu \lambda^{d-1} \Bdm}{d} h^d \le |P_{K\sigma}| \le \frac{2^{1-d}\Bdm}{d} h^d$.\smallskip
\end{enumerate}
Moreover, the dual mesh $\D$ associated to $\M$ satisfies the following properties
\smallskip
\begin{enumerate}[label={\normalfont (CR.{\alph*})}]
\setcounter{enumi}{8}
    \item \label{CR diamdiam} For all $\sigma\in\Eint$, we have $2 \max(\mu,\lambda) h \le \diam(D_\sigma) \le 2 h$, and for all $\sigma\in\Eext$, we have $\max(\mu,2\lambda) h \le \diam(D_\sigma) \le h$.
    \item \label{CR voldiam} For all $\sigma\in\Eint$, we have $\frac{2\mu \lambda^{d-1} \Bdm}{d} h^d \le |D_\sigma| \le \frac{2^{2-d}\Bdm}{d} h^d$, and for all $\sigma\in\Eext$, we have $\frac{\mu \lambda^{d-1} \Bdm}{d} h^d \le |D_\sigma| \le \frac{2^{1-d}\Bdm}{d} h^d$. \smallskip
\end{enumerate}
Finally, we control the number of cells and the number of faces as:
\smallskip
\begin{enumerate}[label={\normalfont (CR.{\alph*})}]
\setcounter{enumi}{10}
    \item \label{CR cardinal K} For all $K\in\T$, we have $ \frac{d \mu^d \Bd }{2^{1-d} \Bdm} \le \# \E_K \le \frac{d \Bd }{2^d \mu \lambda^{d-1} \Bdm}$.
    \item \label{CR cardinal} $\frac{|\Omega|}{\Bd} 2^d h^{-d} \le \# \T \le \frac{|\Omega|}{\Bd} \frac1{\mu^d} h^{-d}$ and $ \frac{d |\Omega|}{2^{2-d} \Bdm} h^{-d} \le \# \E \le \frac{d |\Omega|}{\mu \lambda^{d-1} \Bdm} h^{-d}$.
\end{enumerate}
\end{lemma}

\begin{proof}
    First, let us observe that $H_\sigma\cap B(x_K,\mu h) = \emptyset$ by star-shapedness of $K$ with respect to this ball \ref{HR a volumes}. This implies the first inequality of~\emph{\ref{CR distK}}; the other inequalities of~\emph{\ref{CR distK}} are immediate by definition of $h$ in~\eqref{def diameter}. The second inequality of~\emph{\ref{CR dsig}} is also a consequence of~\ref{HR a volumes}; the third inequality is a consequence of the triangle inequality and the definition of $h$, and the other inequalities of~\emph{\ref{CR dsig}} derive from~\emph{\ref{CR distK}}.
    The inequalities~\emph{\ref{CR diamcell}} come from~\ref{HR a volumes} and the definition of $h$.
    Then, the inequalities~\emph{\ref{CR volcell}} are consequences of~\ref{HR a volumes} (left inequality) and the isodiametric inequality $2^{d}|K| \le \Bd\, h_K^{d}$ (right inequality). The left inequality in~\emph{\ref{CR diamsig}} is a consequence of~\ref{HR b faces} and the isodiametric inequality $2^{d-1}|\sigma| \le \Bdm\, h_{\sigma}^{d-1}$, and the right inequality is a consequence of the definition of $h$. The left inequality in~\emph{\ref{CR volsig}} is~\ref{HR b faces} and the right inequality is a consequence of the isodiametric inequality and the definition of $h$.

   Then, since any cell $K$ is star-shaped with respect to its center by~\ref{HR a volumes}, the pyramids $P_{K\sigma}$ for $\sigma\in \E_K$ are disjoint. Since their closures cover $\overline{K}$ by boundedness of $K$, this gives the partition. Then, since the pyramid $P_{K\sigma}$ is included in the cell $K$, we have $\diam(P_{K\sigma})\le h$, which gives the right-inequality in~\emph{\ref{CR diampyr}}. The left inequality is a consequence of the definition of the pyramid and the fact that $\dd(\sigma,x_K)\ge \mu h$ in \emph{\ref{CR distK}} and the left inequality in \emph{\ref{CR diamsig}}. For~\emph{\ref{CR volpyr}}, we compute the volume of the pyramid $P_{K\sigma}$ as
\begin{equation*}
    \left|P_{K\sigma}\right| = \frac{\dd(x_K,H_\sigma) |\sigma|}{d},
\end{equation*}
that we estimate with~\emph{\ref{CR distK}} and \emph{\ref{CR volsig}}. The inequalities in~\emph{\ref{CR diamdiam}} and~\emph{\ref{CR voldiam}} are obtained similarly, using in addition the lower bound on $\dd_\sigma$ from~\emph{\ref{CR dsig}} for $\Eint$ and the fact that the pyramids are disjoint.

Finally, using again that the pyramids $P_{K\sigma}$ for $\sigma\in\E_K$ form a partition of the cell $K$, we have $|K|=\sum_{\sigma \in\E_K} |P_{K\sigma}|$, so that combining~\emph{\ref{CR volpyr}} and~\emph{\ref{CR volcell}} gives~\emph{\ref{CR cardinal K}}. Using similarly that both the mesh and the dual mesh give a partition of the domain, we have that
\begin{equation*}
     \sum_{K\in\T} |K| = |\Omega| = \sum_{\sigma \in\E} |D_\sigma|,
\end{equation*}
so that using the volume bounds of~\emph{\ref{CR volcell}} and~\emph{\ref{CR voldiam}} gives~\emph{\ref{CR cardinal}}.
\end{proof}

We furthermore discuss the definitions of our discrete gradient and Sobolev seminorm. First, we have the following consistency property.

\begin{lemma}[Weak consistency of discrete gradients]\label{lem.weakconsist} Let $(\M_n)_{n\in\N}$ be a sequence of meshes in the sense of {\normalfont \ref{HM a partition}--\ref{HM e cell centers}} with mesh size $(h_n)_{n\in\N}$. Then for any $u_n\in\R^{\T_n}$, $\varphi\in \left(\mathcal{C}_c^1(\Omega)\right)^d$ and any $1\leq p\leq \infty$
\begin{equation}
\left|\int_\Omega \nabla^{\M_n}u_n\cdot\varphi + (u_n^{\M_n})\nabla\cdot\varphi\right|\leq h_n|\Omega|^{1-\frac1p}\|\nabla\varphi\|_{\mathcal{C}(\overline{\Omega})}|u_n|_{1,p}.
\end{equation}
In particular, if $\lim_{n\to\infty} h_n = 0$ and $\sup_{n\in\N}|u_n|_{1,p}<+\infty$ then the left-hand side tends to $0$ as $n\to\infty$. 
\end{lemma}
\begin{proof}
The divergence theorem shows 
\[
\int_\Omega \nabla^{\M_n}u_n\cdot\varphi + (u_n^{\M_n})\nabla\cdot\varphi = \sum \limits_{\substack{\sigma\in\Eint\\ \sigma\in\EKL}}|\sigma|\frac{u_{n,L}-u_{n,K}}{|D_\sigma|}I_\sigma(\varphi)|D_\sigma|,
\]
with, for $\sigma\in\EKL$, 
\[I_\sigma(\varphi) = \frac{1}{|D_\sigma|} \int_{D_\sigma}\varphi\cdot \nKs - \frac{1}{|\sigma|} \int_{\sigma}\varphi\cdot \nKs = \frac{1}{|D_\sigma|} \frac{1}{|\sigma|}\iint_{D_\sigma \times \sigma} \left(\varphi(x) -\varphi(s)\right) \cdot \nKs \, \dd s \, \dd x
,\]
and since $|I_\sigma(\varphi)|\leq \|\nabla\varphi\|_{\mathcal{C}(\overline{\Omega})}h_n$, the result follows by H\"older inequality.
\end{proof}
Observe that the regularity properties  {\normalfont \ref{HR a volumes}--\ref{HR b faces}} are not required in the previous lemma. Let us note however that in practice the property $\sup_{n\in\N}|u_n|_{1,p}<+\infty$ and compactness of the sequence $(u_n)_n$ will be obtained for a typical finite volume discretization under  {\normalfont \ref{HR a volumes}--\ref{HR b faces}}. Moreover, an additional orthogonality assumption on the mesh is required to ensure the consistency of fluxes for isotropic elliptic or parabolic equations. Finally, let us also recall that strong consistency of the gradient is false in general. More precisely, if $u$ is a smooth function and $\Pi^{\M_n}u$ is its constant by cell projection on the mesh, $\nabla^{\M_n}\Pi^{\M_n}u$ does not converge strongly to $\nabla u$ in most cases.

The Sobolev seminorm is sometimes defined  alternatively without introducing a discrete gradient as follows
\begin{equation*}
N_{1,p}(u) = \left\{\begin{aligned}&\left(\sum\limits_{\substack{\sigma\in\Eint\\ \sigma\in\EKL}}|\sigma|\dd_\sigma^{1-p}|u_L-u_K|^p\right)^{1/p}&& \text{for } 1\leq p< \infty, \\
&\max_{\substack{\sigma\in\Eint\\ \sigma\in\EKL}}\frac{|u_L-u_K|}{\dd_\sigma}&& \text{for } p=\infty,\end{aligned}\right.
\end{equation*}
and, in the homogeneous Dirichlet boundary condition setting,
 \begin{equation*}
N_{1,p,\Gamma_0}(u) = \left\{\begin{aligned}&\left(\sum\limits_{\substack{\sigma\in\Eint\\ \sigma\in\EKL}}|\sigma|\dd_\sigma^{1-p}|u_L-u_K|^p + \sum\limits_{\substack{\sigma\in\Eextz\\ \sigma\in\E_K}}|\sigma|\dd_\sigma^{1-p}|u_K|^p\right)^{1/p}&& \text{for } 1\leq p< \infty, \\
&\max\left(\max_{\substack{\sigma\in\Eint\\ \sigma\in\EKL}}\frac{|u_L-u_K|}{\dd_\sigma}, \max_{\substack{\sigma\in\Eextz\\ \sigma\in\E_K}}\frac{|u_K|}{\dd_\sigma}\right)&& \text{for } p=\infty.\end{aligned}\right.
\end{equation*}

We refer, for instance, to related works~\cite{Herbin95,EGH00,GlitzkyGriepentrog10,BCCHF}. Usual estimates for  TPFA finite volume discretizations involve the quantities $N_{1,p}$ rather than $|\cdot|_{1,p}$. The following lemma shows the equivalence of the two definitions on regular meshes.

\begin{lemma}[Equivalent seminorms]\label{lem:equivseminorm}
 Let $\M$ be a mesh in the sense of {\normalfont \ref{HM a partition}--\ref{HM e cell centers}} satisfying the regularity assumption {\normalfont \ref{HR a volumes}} and let $\Gamma_0\subset \partial \Omega$ be a portion of boundary such that $\Gamma_0=\cup_{\sigma \in \Eextz} \sigma$ for some $\Eextz\subset \Eext$. Then, for any $u\in\R^{\T}$ and $1\leq p\leq \infty$
\[
d^{1-\frac1p} N_{1,p}(u) \leq |u|_{1,p}\leq \left(\frac{d}{\mu}\right)^{1-\frac1p}  N_{1,p}(u),
\]
and
\[
d^{1-\frac1p} N_{1,p,\Gamma_0}(u) \leq |u|_{1,p,\Gamma_0}\leq \left(\frac{d}{\mu}\right)^{1-\frac1p}  N_{1,p,\Gamma_0}(u).
\]
The two left inequalities are equalities if the mesh satisfies additionally the orthogonality assumption: $[x_K,x_L]$ is orthogonal to $\sigma$ for all $\sigma\in\EKL$. 
\end{lemma}
\begin{proof}
For $K,L\in\M$ and $\sigma\in\EKL$, we have that
\[
\mu \dd_\sigma \le \dd(x_K,H_\sigma)+\dd(x_L,H_\sigma)\le \dd_\sigma,
\]
where the left inequality comes from the combination of the first and the last inequality for $\Eint$ in~\ref{CR dsig}, and the right inequality is the second inequality of~\ref{CR dsig}. Therefore,
\[
\frac{\mu}{d}|\sigma|\dd_\sigma\leq |D_\sigma|\leq \frac{1}{d}|\sigma|\dd_\sigma,
\]
which also holds straightforwardly for $\sigma\in \Eext$. The result follows.
\end{proof}

\begin{remark}\label{rk:equivseminorms} Note furthermore that for $p=1$ we have by definition $N_{1,1} = |\cdot|_{1,1}$ and $N_{1,1,\Gamma_0} = |\cdot|_{1,1,\Gamma_0}$.
\end{remark}

We conclude this section with an example of a configuration such that $x_K\in\widehat{\sigma}$ for some cell $K$ and some face $\sigma$, illustrated in Figure~\ref{fig:diamonds}.

\begin{figure}
    \centering
    \includegraphics[width=0.49\linewidth]{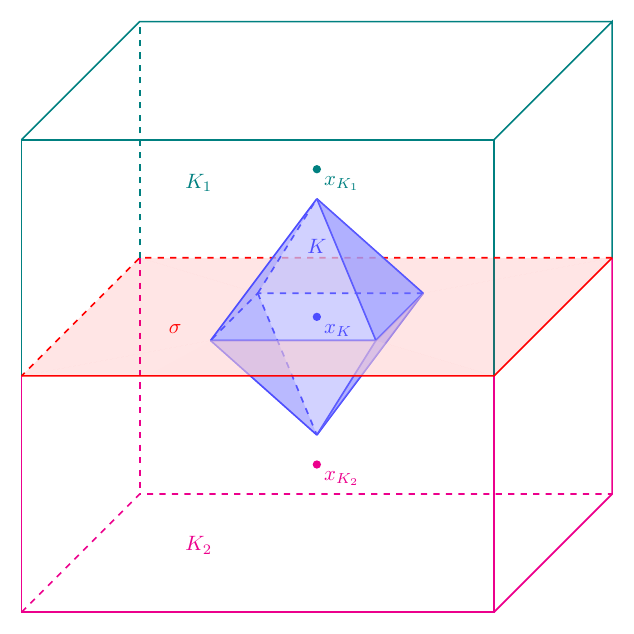}
    \caption{In this situation, $\sigma$ is a rectangle minus a smaller embedded rectangle. Its convex hull $\widehat{\sigma}$, the whole larger rectangle, contains $x_K$.}
    \label{fig:diamonds}
\end{figure}

\end{appendix}

\bigskip

\noindent{\bf Acknowledgements.} M.H. acknowledges support from the CDP C2EMPI, together with the French State under the France-2030 programme, the University of Lille, the Initiative of Excellence of the University of Lille, the European Metropolis of Lille for their funding and support of the R-CDP-24-004-C2EMPI project. The authors acknowledge support from the LabEx CIMI (ANR-11-LABX-0040) as well as the PEPS JCJC 2025 grant from INSMI (CNRS). The authors warmly thank Franck Boyer for sharing Python scripts related to mesh generation and post-processing and Simon Lemaire for many enlightening discussions on mesh regularity and high-order methods.

Claude AI (Anthropic) was used to assist in the improvement of the graphical illustrations in this paper, as well as for final proofreading.

\bibliographystyle{plain}
\bibliography{bibli}

\end{document}